\documentclass[10pt,oneside,a4paper,reqno]{amsart}  
\usepackage{mathrsfs}
\usepackage{amssymb}

\usepackage{amssymb, bbm,
enumerate}
\usepackage{geometry}                
\usepackage[T1]{fontenc}

\usepackage{mathrsfs} 

\usepackage[english]{babel}
\usepackage{graphicx}
\usepackage{color}
\usepackage[colorlinks=true, linkcolor=magenta, citecolor=cyan, urlcolor=blue]{hyperref}

\usepackage{scalerel,stackengine}
\stackMath
\newcommand\reallywidehat[1]{%
\savestack{\tmpbox}{\stretchto{%
  \scaleto{%
    \scalerel*[\widthof{\ensuremath{#1}}]{\kern-.6pt\bigwedge\kern-.6pt}%
    {\rule[-\textheight/2]{1ex}{\textheight}}
  }{\textheight}%
}{0.5ex}}%
\stackon[1pt]{#1}{\tmpbox}%
}
\renewcommand{\phi}{\varphi}

\renewcommand{\Re}{\textup{Re }}
\renewcommand{\Im}{\textup{Im }}

\newcommand{\meanv}[1]{\langle#1\rangle}

\newcommand{\mb}[1]{\mathbb{#1}}

\def\be{\begin{equation}}
\def\ee{\end{equation}}
\def\bea{\begin{eqnarray}}
\def\eea{\end{eqnarray}}
\def\ni{\noindent}
\def\nn{\nonumber}
\def\T{\mathbb{T}}

\def\C{\mathbb{C}}
\def\Z{\mathbb{Z}}
\def\N{\mathbb{N}}
\def\B{\mathscr{B}}

\DeclareMathSymbol{\leqslant}{\mathalpha}{AMSa}{"36} 
\DeclareMathSymbol{\geqslant}{\mathalpha}{AMSa}{"3E} 
\DeclareMathSymbol{\eset}{\mathalpha}{AMSb}{"3F}     
\renewcommand{\leq}{\;\leqslant\;}                   
\renewcommand{\geq}{\;\geqslant\;}                   

\DeclareMathOperator{\Span}{span}

\def\ie{\textit{i.e. }}

\def\a{\alpha}
\def\e{\varepsilon}
\def\d{\delta}
\def\g{\gamma}

\def\D{\Delta}
\def\l{\lambda}
\def\r{\rho}
\def\s{\sigma}
\def\t{\tau}
\def\k{\kappa}

\def\R{\mathbb{R}}
\def\C{\mathbb{C}}

\theoremstyle{plain}
\newtheorem{theorem}{Theorem}[section]
\newtheorem{lemma}[theorem]{Lemma}
\newtheorem{proposition}[theorem]{Proposition}

\theoremstyle{definition}

\theoremstyle{remark}
\newtheorem{remark}[theorem]{Remark}

\numberwithin{equation}{section}

\definecolor{light}{gray}{.9}

\author{Giuseppe Genovese}
\address{Institute of Mathematics, University of Zurich, Winterthurerstrasse 190, CH-8057 Zurich, Switzerland}
\email{giuseppe.genovese@math.uzh.ch}

\author{Renato Luc\`a}
\address{BCAM - Basque Center for Applied Mathematics, 48009 Bilbao, Spain and Ikerbasque, Basque Foundation
for Science, 48011 Bilbao, Spain.}
\email{rluca@bcamath.org} 

\author{Riccardo Montalto}
\address{Dipartimento di Matematica Federigo Enriques, Universit\'a Statale di Milano, Via Saldini 50, 20133, Milan, Italy }
\email{riccardo.montalto@unimi.it}

\title[Quasi-invariance under the Birkhoff map]
{Transformation of the Gibbs measure of the cubic NLS and fractional NLS under an approximated Birkhoff map\footnote{dedicated to the memory of Thomas Kappeler}}

\makeindex
\begin{document}
\begin{abstract}
We study the Gibbs measure associated to the periodic cubic nonlinear Schr\"odinger equation.  We establish a change of variable formula for this measure under the first step of the Birkhoff normal form reduction. We also consider the case of fractional dispersion.
\end{abstract}

\maketitle




\section{Introduction}

We study the cubic non-linear Schr\"odinger equation with fractional dispersion $\alpha >0$
\begin{equation}\label{main equation}
\partial_t u =i\Big(   |D_x|^{2\alpha} u + \sigma |u|^2 u \Big), \quad x \in \T := \R/ (2 \pi \Z)\,,
\end{equation}
where $\sigma= \pm 1$, depending on the defocusing or focusing character of the equation: for $\sigma = 1$ the equation \eqref{main equation} is defocusing, whereas for $\sigma = - 1$ it is focusing. The operator $|D_x|^\a$ is the Fourier multiplier defined by $|D_x|^\a (e^{i n x}) = |n|^\a e^{i n x}$, $n \in \Z$. One is typically interested in the regime $\frac12 < \alpha \leq 1$, being $\alpha =1$ the
usual cubic NLS equation and $\alpha = \frac12$ the cubic half wave equation. The main results of this paper are for $\a\in[\bar{\alpha},1]$ where $\bar{\alpha} := \frac{1+ \sqrt{97}}{12} \sim 0.9$. 
 
As $\a$ varies the equation (\ref{main equation}) describes a Hamiltonian (for $\a=1$ in fact integrable) PDE with energy 
\begin{equation}\label{Hamiltoniana}
H^{(\a)}(u) :=  \int_\T ||D_x|^\alpha u|^2 + \frac{\sigma}2\int_\T |u|^4\,dx \,.
\end{equation}
To these energies one associates infinite dimensional Gibbs measures $\r_\a$ absolutely continuous w.r.t. the Gaussian measures $\tilde\g_{\alpha}(du)$ restricted to some $L^2$-ball:
\be\label{eq:Gibbs}
\r_\a(du):=e^{-\frac{\sigma}2\|u\|_{L^4}^4}\tilde\g_{\alpha}(du)\,; 
\ee
see \eqref{eq:def-gammatilde} below and the surrounding discussion for the definition of $\tilde\g_{\alpha}(du)$. 
These are central objects in this note. The construction of $\r_1$ was achieved in the seminal papers of Lebowitz-Rose-Speer \cite{LRS} and Bourgain \cite{B94}. For fractional $\alpha$ one can follow the same procedure, the only delicate point being the integrability of the term 
$
e^{-\frac{\sigma}{2}\|u\|^4_{L^4}}
$
w.r.t. $\tilde\g_\a$ in the focusing case. This is achieved in the subsequent Proposition~\ref{prop:exp-mom-L4}. 

The aim of this paper is to study how the measures $\r_\a$ transform under the action of a given canonical transformation which removes the non-resonant part of the Hamiltonian \eqref{Hamiltoniana} up to terms of order $|u|^6$. We call the reduced Hamiltionian {\em Birkhoff normal form} and the reducing transformation {\em approximate Birkhoff map}. For general Hamiltonian systems a classical theorem of Birkhoff establishes the existence of a canonical transformations putting the Hamiltonian in normal form up to a remainder of a given arbitrary order (see for instance \cite[Theorem G.1]{thomasbook}). The construction of such maps has been exploited in infinite dimension for many Hamiltonian PDEs in different contexts, Starting from the pioneering papers \cite{Bou96}, \cite{Bam03}, \cite{BG06}. Without trying to be exaustive, we also mention several more recent extensions to PDEs in higher space dimension and to quasi-linear PDEs: see \cite{Delort-2009}, \cite{Delort-2015}, \cite{Delort-Imrekaz}, \cite{BDGS}, \cite{BD}, \cite{BeFePus}, \cite{BMM}, \cite{Ionescu-Pusateri}, \cite{FeMo}, \cite{StaffWilson20}, \cite{BFM23}. 
 
One nice feature of the approximate Birkhoff map is that it can be expressed as a Hamiltonian flow. Let us introduce it. 
Consider
$$\Phi_t^N : E_N \to E_N\,, \quad E_N := {\rm span}\big\{ e^{i j x} : |j| \leq N \big\}$$ 
defined by the system of ODEs
\be
\frac{d}{dt}(\Phi_t^N(u))(n)=\sum_{\substack{
|j_1|, |j_2|, |j_3| \leq N \\
j_1 + j_2 - j_3 = n\\
|j_1|^{2\alpha} + |j_2|^{2\alpha} - |j_3|^{2\alpha} - |n|^{2\alpha} \neq 0
}} \dfrac{- \sigma}{ \Big( |j_1|^{2\alpha} + |j_2|^{2\alpha} - |j_3|^{2\alpha} - |n|^{2\alpha} \Big)}  u(j_1)  u(j_2) \overline u(j_3)\,.
\ee
We are interested in the Birkhoff map/flow $\Phi_t:=\Phi_t^{\infty}$ and, more specifically, in the 
the $1$-time shift $\Phi_1$, that we abbreviate to 
$\Phi$ in order to simplify the notations. Clearly the Birkhoff map also depends on $\a$, but for notational simplicity we will not keep track of that in the manuscript. This transformation acts on the Hamiltonian as follows:
\be\label{eq:birkohoff=energia0}
H^{(\a)} [\Phi u] = \||D_x|^{\alpha} u\|^2_{L^2}+ \frac{\sigma}2\|u\|^4_{L^2}+R^{(\a)}[u]
\ee
with $R^{(\a)}[u] = O(|u|^6)$ is a remainder which has a zero of order six at $u = 0$. This identity can be 
easily justified for sufficiently regular functions, for instance for $u \in H^{s}$, $s > 1/2$, see e.g. \cite{BG06}.

Next, we shortly introduce $\tilde\g_{\alpha}$. Let $\{g_n\}_{n\in\Z}$ be a sequence of independent, identically distributed complex centred Gaussian random variables with unitary variance. We consider the random Fourier series
\begin{equation}\label{Def:gammaK}
\sum_{n\in\Z} \frac{g_n}{(1+|n|^{2\alpha})^{\frac{1}{2}}}\, e^{inx} \,.
\end{equation}
If $\alpha >1/2$, this defines a function on $L^2(\T)$ for almost all realisation of the sequence $\{g_n\}_{n\in\Z}$. Thanks to separability and the isomorphism between~$\C^{2N+1}$ and 
\be\label{def:EN}
E_N:=\Span_{\mb C}\{e^{inx}\,:\, |n|\leq N\}\,
\ee
the space $L^2(\T)$ inherits the measurable-space structure by a standard limit procedure and we will denote by
$\mathscr{B}(L^2(\T))$ the Borel $\s$-algebra on~$L^{2}(\T)$. The Gaussian measure on $\mathscr{B}(L^2(\T))$ induced by (\ref{Def:gammaK}) is denoted by $\g_{\alpha}$. 
The triple $(L^2(\T), \B(L^2(\T)), \g_{\alpha})$ is a Gaussian probability space satisfying the   
concentration properties:
$$\tilde\g_{\alpha} \bigg( \bigcap_{s < \alpha - \frac{1}{2}}  H^{s}(\T) \bigg)=1 \, ,\quad \tilde\g_{\alpha} \bigg(H^{\alpha - \frac12}(\T) \bigg)=0\,. $$  
The expectation value w.r.t. $\g_\a$ is always indicated by $E_\a$.
Finally we introduce the restriction of $\g_\a$ to a ball of $L^2(\T)$ as
\be\label{eq:def-gammatilde}
\tilde\g_{\alpha}(A):=\g_{\alpha}(A\cap\{ \|u\|_ {L^2}\leq R\}), \qquad A\in\mathscr B(L^2(\T))\,,
\ee
for some $R>0$.
The measure $\tilde\g_\a$ proves useful since the $L^2$ norm is preserved by the equation~(\ref{main equation}). 

Now we are ready to give our main theorem (recall \eqref{eq:Gibbs}).

\begin{theorem}\label{TH:main-alpha}
Let $ \bar{\alpha} := \frac{1+ \sqrt{97}}{12} \sim 0.9$, $\a\in(\bar{\alpha},1]$, $s>2-2\a$ and $A\subset H^s$ mesurable. Then the quantity
\be\label{eq:main}
\int_A \exp \left(- \int_0^1 \!\!\!  \frac{d}{d\tau}\left( H^{(\a)}[ \Phi_\tau(u)]  \right) d\tau \right)
\r_{\a}(du)
\ee
is well defined for all measurable sets $A\in \mathscr B(L^2(\T))\cap H^{s}(\T)$ and equals $\r_\a(\Phi(A))$.
\end{theorem}

The above Theorem can be probably extended to include the case $s=2-2\a$. For simplicity we do not attempt that here, but remark that from our proof it follows that $s\geq0$ if $\a=1$.
 
Equation (\ref{main equation}) is completely integrable for $\a=1$. In the defocusing case, exploiting the complete integrability it is possible to construct a map, so-called {\em global Birkhoff map}, which trivialises equation (\ref{main equation}) and transforms the energies into non-linear Sobolev norms \cite{thomas-NLS}. Our approximate Birkhoff map can be though of as a local approximation of the global Birkhoff map, which is accurate about the origin. Constructing the global Birkhoff map for a Hamiltonian PDE is a very challenging task, which has been achieved only in few cases, namely by Kappeler and P\"oschel for the KdV equation \cite{thomasbook}, by Greb\'ert and Kappeler for the defocussing NLS equation \cite{thomas-NLS}, by G\'erard, Kappeler and Topalov for the Benjamin-Ono equation \cite{thomasBOno1, thomasBOno2}. In the recent preprint \cite{nik2} Tzvetkov exhibits a large class of invariant measures w.r.t. the Benjamin-Ono flow (including Gaussian-based measures), using the global Birkhoff map of G\'erard, Kappeler and Topalov. However these measures are constructed on the image space of the map (that is, they are given in terms of Birkhoff coordinates) and the author conjectures that the Birkhoff map acts on Gibbs measure in a quasi-invariant way (see \cite[Conjecture 3.1]{nik2}). In the same spirit we conjecture that a change of variable formula similar to the one of Theorem \ref{TH:main-alpha} can be proved for map reducing the Hamiltonian in normal forms of any order. We plan to further investigate this topic in future works. 

The proof is based on the Tzvetkov argument for quasi-invariance \cite{sigma} and the method of \cite{BBM1} for the transported density (for other ways of determining the density, see \cite{deb,forl-tol}). In particular we exploit crucially that the approximated Birkhoff map is given by a Hamiltonian flow, a property which is not enjoyed by the global Birkhoff map.

The restriction on $\a$ in Theorem \ref{TH:main-alpha} arises   
from the probabilistic analysis, but it is necessary in order to justify the dynamics as well (see Proposition \ref{prop:alabourgain}).
 We do not assign any special meaning to this range $(\bar{\alpha},1]$, 
$\bar{\alpha} := \frac{1+ \sqrt{97}}{12} \sim 0.9$. We stress that if $\a<1$ showing that the flow of the approximate Birkhoff map is globally well defined is quite a non-trivial task. By a direct application of the Bourgain probabilistic globalisation method \cite{B94} we prove here global well-posedness of the flow-map for almost all data in $H^{(\alpha - 1/2)^-}$ w.r.t. the Gibbs measure; however it is not clear to us how to do that deterministically, since we cannot prove local well-posedness for data in $L^2$ and the $L^2$ norm is the sole conserved quantity at our disposal 
when $\a<1$.

The paper is organised as follows. The Birkhoff normal form reduction is given in Section \ref{section:Normal} along with a crucial estimate on the derivative along the Birkhoff flow of the Hamiltonian at time zero. In Section \ref{section:flowNLS} we give the necessary deterministic estimates on the Birkhoff flow for $\a=1$, while the ones for fractional $\a$ are given in Section \ref{section:flow-fract}. In Section \ref{section:prob} we give the necessary probabilistic estimates. Section \ref{sect:quasi} is devoted to the proof of the quasi-invariance and here we also prove the probabilistic global well-posedness of the Birkhoff flow for $\a > \bar{\alpha}$, $ \bar{\alpha} := \frac{1+ \sqrt{97}}{12} \sim 0.9$. We establish the formula for the transported density in Section \ref{sect:density}.

\subsection*{Notations}

Given a function $u : \T \to \R$, we denote by $f(n)$ its Fourier coefficient 
$$
u(n) := \frac{1}{2 \pi} \int_\T u(x) e^{- i n x}\, d x, \quad n \in \Z\,. 
$$
 We define the Sobolev norms $H^{s}$ of $f$ as 
$$
\| u \|_{H^s}^2 := 
\sum_{n \in \Z}(1+|n|^{2 s})|{u}(n)|^2\,,
$$
We also define the Fourier-Lebesgue norm for any $p \geq 1$
$$
\|u\|_{FL^{0,p}}= \Big( \sum_{n \in \Z}|u(n)|^p \Big)^{\frac1p}\,. 
$$
A ball of radius $R$ and centred in zero in the $H^s$ topology is denoted by $B_s(R)$. We drop the subscript for $s=0$ (ball of $L^2$). We write $\meanv{\cdot} := (1 + |\cdot|^2)^{1/2}$.
We write $X\lesssim Y$ if there is a constant $c>0$ such that $X\leq cY$ and $X\simeq Y$ if $Y\lesssim X\lesssim Y$. We underscore the dependency of $c$ on the additional parameter $a$ writing $X\lesssim_a Y$. $C,c$ always denote constants that often vary from line to line within a calculation. $A+B$ is the Minkowski sum of the sets $A, B$, namely $A + B := \{ x + y : x \in A\,,\, y \in B \}$. 
We denote by $\Pi_N$ the standard orthogonal projection
$$
\Pi_N(u):=\sum_{|n|\leq N} {u}(n)e^{inx}\,,\qquad \Pi_N^\bot := {\rm Id} - \Pi_N\,,
$$
where ${u}(n)$ is the $n$-th Fourier coefficient of $u\in L^2$. 
Also, we denote the Littlewood-Paley projector by $\D_0:=\Pi_{1}$, 
$\D_j:=\Pi_{2^{j}}-\Pi_{2^{j-1}}$, $j \in \N$.

We will use the following well-known tail bounds for sequences of independent centred Gaussian random variables $X_1,\ldots,X_d$ (see for instance \cite{ver}): the Hoeffding inequality
\be\label{eq_Hoeffing}
P\left(\sum_{i=1}^d |X_i| \geq \lambda \right)\leq C\exp\left(-c\frac{\lambda^2}{d}\right)\,
\ee
and the Bernstein inequality
\be\label{eq_Bernstein}
P\left( \left|\sum_{i=1}^d |X_i|^2-E[\sum_{i=1}^d |X_i|^2]\right| \geq \lambda \right)\leq C\exp\left(-c\min\left(\l,\frac{\lambda^2}{d}\,\right)\right)\,.
\ee

\subsection*{Acknowledgements} This work is dedicated to the memory of our colleague and dearest friend Thomas Kappeler. Discussing with him through the past years greatly inspired us to start the study of the invariance and quasi-invariance properties of Birkhoff maps in the context of dispersive PDEs, of which this paper is the first step.

We thank Nikolay Tzvetkov for a stimulating discussion.

Renato Lucà is supported by the Basque Government under the (BCAM) BERC program 2022-2025 and by the Spanish Ministry of Science, Innovation and Universities through the (BCAM) Severo Ochoa accreditation CEX2021-001142-S and the project PID2021-123034NB-I00. Renato Lucà is also supported by the Ramon y Cajal fellowship RYC2021-031981-I.

Riccardo Montalto is supported by the ERC STARTING GRANT "Hamiltonian Dynamics, Normal Forms and Water Waves" (HamDyWWa), project number: 101039762. Views and opinions expressed are however those of the authors only and do not necessarily reflect those of the European Union or the European Research Council. Neither the European Union nor the granting authority can be held responsible for them. 

Riccardo Montalto is also supported by INDAM-GNFM.


\section{The Birkhoff Normal form transformation}\label{section:Normal}

In this section we introduce the finite dimensional approximated Birkhoff transformation studied in this paper.
To do so, we present some more notations.

Given the Hamiltonian function $\mathcal{F}$ the associated Hamilton equations write as
$$
\partial_t U = X_{\mathcal{F}}(U) 
$$
where $U := (u, \overline u) \in {\mathcal L}^2(\T) := L^2(\T) \times L^2 (\T)$ and the Hamiltonian vector field $X_{\mathcal{F}}$ is given by 
\begin{equation}\label{campo hamiltoniano NLS}
\begin{aligned}
X_{\mathcal{F}}(U)  := \begin{pmatrix}
i \nabla_{\overline u} {\mathcal{F}}(U) \\
- i \nabla_u {\mathcal F}(U)
\end{pmatrix} = {J} \nabla_U {\mathcal F}(U)\,, \\
\nabla_U {\mathcal F}(U) = (\nabla_u {\mathcal F}(U), \nabla_{\overline u} {\mathcal F}(U)), \quad {J} := \begin{pmatrix}
0 & i \\
- i & 0
\end{pmatrix}\,. 
\end{aligned}
\end{equation}
We also define, for any $s \geq 0$, $\mathcal{H}^s(\T) := H^s(\T) \times H^s(\T)$. Given two Hamiltonian functions ${\mathcal F}, { \mathcal G} : {\mathcal L}^2(\T) \to \R$, we define the Poisson bracket
\begin{equation}\label{def Poisson}
\{ { \mathcal F}, { \mathcal G} \} (U) := D_U { \mathcal F}(U)[{ J \nabla { \mathcal G}(U)}]\,. 
\end{equation}
Given ${\mathcal G} : {\mathcal L}^2(\T) \to \R$ such that its Hamiltonian vector field $X_{\mathcal F} : {\mathcal H}^s(\T) \to {\mathcal H}^s(\T)$, $s \geq 0$, we denote by $\Phi_t^{\mathcal F}$ its Hamiltonian flow, namely 
\begin{equation}\label{flusso hamiltoniano}
\begin{cases}
\partial_t \Phi^{\mathcal F}_t(U) = X_{\mathcal F}(\Phi^{\mathcal F}_t(U)) \\
\Phi^{\mathcal F}_0(U) = U\,. 
\end{cases}
\end{equation}
Given a function ${\mathcal G}$, one has that 
\begin{equation}\label{Lie expansion}
{ \mathcal G} \circ \Phi^{\mathcal F}_t = {\mathcal G} + t \{{\mathcal G}, {\mathcal F} \} + {R}_t, \quad {R}_t 
:= \int_0^t (t - \tau) \{ \{{\mathcal G}, {\mathcal F}\}, {\mathcal F} \} \circ \Phi^{\mathcal F}_\tau \, d \tau\,.
\end{equation}
Given a Hamiltonian $\mathcal{F} : {\mathcal H}^s(\T) \to \R$, we define the 
truncated Hamiltonian ${\mathcal F}_N := {\mathcal F}_{| {\mathcal E}_N}$, ${\mathcal E}_N := E_N \times E_N$ (recall \eqref{def:EN}) and the truncated Hamiltonian vector field is given by 
\begin{equation}\label{DefTrunc1}
X_{{\mathcal F}_N}(U) := \Pi_N X_{{\mathcal F}}(\Pi_N U)\,.
\end{equation}
The truncated Hamiltonian flow $\Phi_t^{{\mathcal F}_N} : {\mathcal E}_N \to {\mathcal E}_N$ then solves
\begin{equation}\label{flusso hamiltoniano troncato}
\begin{cases}
\partial_t \Phi_t^{{\mathcal F}_N}(U) = X_{{\mathcal F}_N}(\Phi_t^{{\mathcal F}_N}(U)) \\
\Phi_0^{{\mathcal F}_N}(U) = U\,, 
\end{cases} \quad U \in {\mathcal E}_N\,. 
\end{equation}

Let us now define the Birkhoff map. We set
\begin{equation}\label{formula cal FN}
\begin{aligned}
& {\mathcal F}_N \equiv {\mathcal F}^{(\a)}_N\,, \\
& {\mathcal F}^{(\a)}_N := \!\!\!
\sum_{\substack{
|n_1|,|n_2|, |m_1|,|m_2| \leq N \\
n_1 + n_2 = m_1 + m_2\\
|n_1|^{2\a} + |n_2|^{2\a} \neq |m_1|^{2\a} + |m_2|^{2\a}
}} \dfrac{-  \sigma}{ 2 i \Big( |n_1|^{2\a} + |n_2|^{2\a} - |m_1|^{2\a}- |m_2|^{2\a} \Big)}  u(n_1)  u(n_2) \overline u(m_1) \overline u(m_2)\,.
\end{aligned}
\end{equation}

The flow-map $\Phi_t^N := \Phi^{{\mathcal F}_N}_{t} : {\mathcal E}_N \to {\mathcal E}_N$ defined by the system of ODEs
\be\label{eq:flomap}
\frac{d}{dt}(\Phi_t^N(u))(n)=\sum_{\substack{
|j_1|, |j_2|, |j_3| \leq N \\
j_1 + j_2 - j_3 = n\\
|j_1|^{2\a} + |j_2|^{2\a} \neq |j_3|^{2\a} + |n|^{2\a}
}} \dfrac{- \sigma}{ \Big( |j_1|^{2\a} + |j_2|^{2\a} - |j_3|^{2\a} -| n|^{2\a} \Big)}  u(j_1)  u(j_2) \overline u(j_3)
\ee
is Hamiltonian w.r.t. ${\mathcal F}_N$ (with canonical Poisson brackets). We call this finite dimensional Birkhoff flow and we set  
$$
\Phi^N \:= (\Phi_t^N)_{t = 1}.
$$
Recall that we omit the dependence on $\a$ in the energy and in the Birkhoff flow map.

The following lemma is a classical fact. 
\begin{lemma}\label{lemma equazione omologica}
Let $\a\in(1/2,1]$. The Hamiltonian ${\mathcal F}_N$ satisfies 
\begin{equation}\label{omologica cal FN}
\{ \| |D_x|^{\alpha} \Pi_N u \|_{L^2}\,,\, {\mathcal F}_N(u) \} + \frac{\sigma}2 \| \Pi_N u \|_{L^4}^4 = \sigma \| \Pi_N u \|_{L^2}^4\,. 
\end{equation}
\end{lemma}
\begin{proof}
A direct computation gives
\begin{equation}\label{homological equation}
 \{ \||D_x|^{\alpha} \Pi_N u\|^2_{L^2}, {\mathcal F}_N \} + \frac{\sigma}2\|\Pi_Nu\|^4_{L^4} = \frac{\sigma}2 \sum_{\begin{subarray}{c}
|j_1|, \ldots, |j_4| \leq N \\
j_1 + j_2 = j_3 + j_4 \\
|j_1|^{2\a} + |j_2|^{2\a} = |j_3|^{2\a} + |j_4|^{2\a}
\end{subarray}}  u(j_1)  u(j_2) \overline u(j_3) \overline u(j_4)\,.
\end{equation}
For $\a=1$ we note that under the constraint $j_1 + j_2 = j_3 + j_4$, one obtains that 
$$
\begin{aligned}
j_1^2 + j_2^2 - j_3^2 - j_4^2 
& = 2 (j_3 - j_2)(j_1 - j_3)\,. 
\end{aligned}
$$
Therefore if $j_1^2 + j_2^2 - j_3^2 - j_4^2 = 0$, then either $j_2 = j_3$ or $j_1 = j_3$. If $j_2 = j_3$, then $j_1 + j_2 - j_3 - j_4 = 0$ gives $j_1 = j_4$. Similarly if $j_1 = j_3$, then $j_2 = j_4$. The same extends to fractional $\a\in(1/2,1]$ thanks to \cite[Lemma 2.4]{Forlano}. Indeed, this lemma implies that if $j_1 + j_2 - j_3 - j_4 = 0$, one has the lower bound
$$
\begin{aligned}
& \Big||j_1|^{2 \alpha} + |j_2|^{2 \alpha} - |j_3 |^{2 \alpha} - |j_4|^{2 \alpha} \Big| \geq C |j_1 - j_4| |j_2 - j_4| j_{\max}^{2 \alpha - 2} \\
& j_{\max} := {\rm max}\{ |j_1|, |j_2|, |j_3|, |j_4| \}
\end{aligned}
$$
for some constant $C > 0$. This latter bound implies that if $|j_1|^{2 \alpha} + |j_2|^{2 \alpha} - |j_3 |^{2 \alpha} - |j_4|^{2 \alpha} = 0$, then $\{ j_1, j_2 \} = \{ j_3, j_4 \}$. Therefore the r.h.s. of \eqref{homological equation}
takes the form 
\begin{equation}\label{forma cal ZN}
\sigma \sum_{|j|, |j'| \leq N} |u(j)|^2 |u(j')|^2= \sigma \|\Pi_Nu\|^4_{L^2}\,. 
\end{equation}
\end{proof}

\section{The flow of the Birkhoff map for the cubic NLS}\label{section:flowNLS}

Here we analyse the well-posedness of the Birkhoff flow. The case $\a=1$ is easier to deal with and will be presented first. We first provide some estimates on the Hamiltonian vector field $X_{{\mathcal F}_N}$. 

\begin{lemma}\label{prop campo hamiltoniano X FN}
Let $\a=1$ and $N \in \N \cup \{ \infty \}$.

$(i)$ For any $\sigma \geq 0$ and for any $u \in H^\sigma(\T)$ it is
\begin{equation}\label{RickBoundE1}
\| X_{{\mathcal F}_N}(u) \|_{H^\s} \lesssim \| u \|_{L^2}^2 \| u \|_{H^\sigma}\,. 
\end{equation}
$(ii)$ For any $\sigma \geq 0$ and for any $u, v \in H^\sigma (\T)$ it is
$$
\|  X_{{\mathcal F}_N}(u) - X_{{\mathcal F}_N}(v) \|_{H^\s} \lesssim 
\|  u - v \|_{H^\s} ( \| u \|_{H^\s}^2 + \| v \|_{H^\s}^2) .
$$
\end{lemma}
\begin{proof}
The Hamiltonian vector field is $X_{{\mathcal F}_N} = (i \nabla_{\bar u} {\mathcal F}_N, - i \nabla_u {\mathcal F}_N)$ and one computes, by recalling formula \eqref{formula cal FN}, that 
$$
\begin{aligned}
\partial_{\bar u(n) } {\mathcal F}_N(u) & = \sum_{\substack{
|j_1|, |j_2|, |j_3| \leq N \\
j_1 + j_2 - j_3 = n\\
j_1^2 + j_2^2 \neq j_3^2 + n^2
}} \dfrac{\sigma i }{2 \Big( j_1^2 + j_2^2 - j_3^2 - n^2 \Big)}  u(j_1)  u(j_2) \overline u(j_3) \\
& = \sum_{\substack{
|j_1|, |j_2|, |j_3| \leq N \\
j_1 + j_2 - j_3 = n\\
}} \dfrac{ \sigma i }{2 (j_3 - j_2)(j_1 - j_3)}  u(j_1)  u(j_2) \overline u(j_3) \,.
\end{aligned}
$$
Then in order to deduce the desired estimates in $(i)$-$(ii)$, it is enough to prove that the trilinear form ${\mathcal T}$ defined by 
\begin{equation}\label{capolavoro elegante}
\begin{aligned}
& {\mathcal T}[u, v, \varphi] := \sum_{n \in \Z} {\mathcal T}_n[u, v, \varphi] e^{i n x}\,, \\
& {\mathcal T}_n[u, v, \varphi] := \sum_{\substack{
|j_1|, |j_2|, |j_3| \leq N \\
j_1 + j_2 - j_3 = n\\
}} \dfrac{ \sigma i }{2 (j_3 - j_2)(j_1 - j_3)}  u(j_1)  v(j_2) \varphi(j_3) 
\end{aligned}
\end{equation}
is continuous on $H^\sigma(\T)$ and satisfies 
\begin{equation}\label{prop cal T}
\begin{aligned}
\| {\mathcal T}[u,v, \varphi] \|_{H^\s} & \lesssim_\sigma \| u \|_{H^\sigma} \| v \|_{L^2} \| \varphi\|_{L^2} + \| u \|_{L^2} \| v \|_{H^\sigma} \| \varphi \|_{L^2 } + \| u \|_{L^2} \| v \|_{L^2} \| \varphi \|_{H^\s} \\
& \forall u\,,\, v\,,\, \varphi \in H^\sigma(\T)\,. 
\end{aligned}
\end{equation}One has 
\begin{equation}\label{meccanica celeste m - 1}
\begin{aligned}
\| {\mathcal T}[u, v, \varphi] \|_{H^\s}^2 & = \sum_{n \in \Z} \langle n \rangle^{2 \sigma} | {\mathcal T}_n[u, v, \varphi]|^2 \lesssim  \sum_{n \in \Z} \Big(  \sum_{\substack{
|j_1|, |j_2|, |j_3| \leq N \\
j_1 + j_2 - j_3 = n\\
}} \dfrac{ \langle n \rangle^\sigma }{|j_3 - j_2| |j_1 - j_3|}  |u(j_1) | |v(j_2) | |\varphi(j_3) | \Big)^2 \,.
\end{aligned}
\end{equation}
By using that $n = j_1 + j_2 - j_3$, one gets that $\langle n \rangle^\sigma \lesssim_\sigma \langle j_1 \rangle^\sigma + \langle j_2 \rangle^\sigma + \langle j_3 \rangle^\sigma$, implying that 
\begin{equation}\label{meccanica celeste m}
\begin{aligned}
\| {\mathcal T}[u, v, \varphi] \|_{H^\s}^2 & \lesssim_\sigma T_1 + T_2 + T_3\,, \\
& T_1 :=  \sum_{n \in \Z} \Big(  \sum_{\substack{
|j_1|, |j_2|, |j_3| \leq N \\
j_1 + j_2 - j_3 = n\\
}} \dfrac{ 1 }{|j_3 - j_2| |j_1 - j_3|}  \langle j_1 \rangle^\sigma |u(j_1) | |v(j_2) | |\varphi(j_3) | \Big)^2  \\
& T_2 :=     \sum_{n \in \Z} \Big(  \sum_{\substack{
|j_1|, |j_2|, |j_3| \leq N \\
j_1 + j_2 - j_3 = n\\
}} \dfrac{ 1 }{|j_3 - j_2| |j_1 - j_3|}  |u(j_1) \langle j_2 \rangle^\sigma | |v(j_2) | |\varphi(j_3) | \Big)^2  \\
& T_3 :=    \sum_{n \in \Z} \Big(  \sum_{\substack{
|j_1|, |j_2|, |j_3| \leq N \\
j_1 + j_2 - j_3 = n\\
}} \dfrac{ 1 }{|j_3 - j_2| |j_1 - j_3|}  |u(j_1) | |v(j_2) | \langle j_3 \rangle^\sigma |\varphi(j_3) | \Big)^2 \,. \\
\end{aligned}
\end{equation}
The estimate of $T_1, T_2, T_3$ is similar, hence we concentrate on the estimate for $T_1$. 
By using the Cauchy Schwartz inequality, one obtains that for any $n \in \Z$
\begin{equation}\label{meccanica celeste m 2}
\begin{aligned}
& \sum_{\substack{
j_1, j_2 \in \Z
}} \dfrac{ 1 }{\langle j_1 - n \rangle \langle j_2 - n \rangle} \langle j_1 \rangle^\sigma |u(j_1) | |v(j_2) | |\varphi(j_1 + j_2 - n) |  \\
& \quad \lesssim \Big( \sum_{\substack{
j_1, j_2 \in \Z
}} \langle j_1 \rangle^{2 \sigma} |u(j_1) |^2 |v(j_2) |^2 |\varphi(j_1 + j_2 - n) |^2 \Big)^{\frac12} \Big( \sum_{\substack{
j_1, j_2 \in \Z
}} \dfrac{ 1 }{\langle j_1 - n \rangle^2 \langle j_2 - n \rangle^2}  \Big)^{\frac12} \\
& \quad \lesssim  \Big( \sum_{\substack{
j_1, j_2 \in \Z
}} \langle j_1 \rangle^{2 \sigma} |u(j_1) |^2 |v(j_2) |^2 |\varphi(j_1 + j_2 - n) |^2 \Big)^{\frac12} \Big( \sum_{\substack{
k_1, k_2 \in \Z
}} \dfrac{ 1 }{\langle k_1 \rangle^2 \langle  k_2 \rangle^2}  \Big)^{\frac12}  \\
& \quad \lesssim \Big( \sum_{\substack{
j_1, j_2 \in \Z
}} \langle j_1 \rangle^{2 \sigma} |u(j_1) |^2 |v(j_2) |^2 |\varphi(j_1 + j_2 - n) |^2 \Big)^{\frac12}\,.
 \end{aligned}
\end{equation}
Thus, $T_1$ can be estimated as 
$$
\begin{aligned}
T_1 & \lesssim \sum_{n \in \Z}   \sum_{\substack{
j_1, j_2 \in \Z
}} \langle j_1 \rangle^{2 \sigma} |u(j_1) |^2 |v(j_2) |^2 |\varphi(j_1 + j_2 - n) |^2  \\
& \lesssim \sum_{j_1 \in \Z} \langle j_1 \rangle^{2 \sigma} |u(j_1)|^2 \sum_{j_2 \in \Z} |v(j_2)|^2 \sum_{n \in \Z} |\varphi(j_1 + j_2 - n) |^2 \\
& \stackrel{j_1 + j_2 - n = k}{\lesssim} \sum_{j_1 \in \Z}  \langle j_1 \rangle^{2 \sigma} |u(j_1)|^2 \sum_{j_2 \in \Z} |v(j_2)|^2 \sum_{k \in \Z} |\varphi(k) |^2 \\
& \lesssim \| u \|_{H^\s}^2 \| v \|_{L^2}^2 \| \varphi \|_{L^2}^2 \,.
\end{aligned}
$$
By similar arguments, one can show that 
$$
T_2 \lesssim \| u \|_{L^2}^2 \| v \|_{H^\sigma}^2 \| \varphi \|_{L^2}^2\,, \quad T_3 \lesssim \| u \|_{L^2}^2 \| v \|_{L^2}^2 \| \varphi \|_{H^\s}^2 \,,
$$
implying the estimate \eqref{prop cal T}. 
Hence the items $(i)$ and $(ii)$ follow since $\nabla_{\bar u} {\mathcal F}_N = {\mathcal T}[u, u, \overline u]$ and thus
$$\nabla_{\bar u} {\mathcal F}_N(u) - \nabla_{\bar v} {\mathcal F}_N(v) = {\mathcal T}[u, u, \overline u] -
{\mathcal T}[v, v, \overline v] = 
{\mathcal T}[u-v, u, \overline u] + {\mathcal T}[v, u-v, \overline u] + {\mathcal T}[v, v, \overline{u-v} ]$$
\end{proof}
The foregoing lemma implies that the flow $\Phi^N_t \:= \Phi^{{\mathcal F}_N}_t$ is well defined in $H^\sigma(\T)$ for any $\sigma \geq 0$. This follows by a standard Picard iteration. In particular, the following lemma holds: 

\begin{lemma}\label{LWPlemmaNLS}
Let $N \in \N \cup \{ \infty \}$. For any $u_0 \in L^2(\T)$, there exists a unique local solution $u \in C^1(\R, L^2(\T))$ which solves the Cauchy problem 
$$
\begin{cases}
\partial_t u(t) = X_{{\mathcal F}_N}(u(t)) \\
u(0) = u_0 .
\end{cases}
$$
 Thus, the flow $\Phi^N_t: L^2(\T) \to L^2(\T)$
 is a well defined $C^1$ map. 
 Moreover, for all $\sigma \geq 0$, $u_0 \in {H^\sigma}(\T)$ and all $T>0$ the $H^\sigma$ norm is controlled as 
 $\| u(t) \|_{H^\s} \leq C e^{C T} \| u_0 \|_{H^\s}$, we have 
 $u \in C^1([- T, T], H^{\sigma}(\T))$ and the 
 the flow $\Phi_{{\mathcal F}_N}^t: H^\sigma(\T) \to H^\sigma(\T)$ 
 is a well defined $C^1$ map for any $\tau \in [- T, T]$.
 \end{lemma}

 \begin{proof}
 The local existence follows by a standard fixed point argument on the Volterra integral operator 
 $$
 u(t) \mapsto {\mathcal S}(u)(t) := u_0 + \int_0^t X_{{\mathcal F}_N}(u(\tau))\, d \tau 
 $$
 in the closed ball 
 \begin{equation}\label{BallsBsigma}
 \Big\{ u \in C^0([- \delta, \delta], H^\sigma(\T)) : \| u \|_{L^\infty_T H^\sigma_x} 
 := \sup_{t \in [- \delta, \delta]}\| u(t) \|_{H^\s} \leq R \Big\}\,. 
 \end{equation}
 This fixed point argument requires that $R$ is larger than $\| u_0 \|_{H^\s}$ and $\delta R^2 \ll 1$. If $u$ is a fixed point for ${\mathcal S}$ one also immediately gets that $u \in C^1([- \delta, \delta], H^\sigma(\T))$.  For $L^2$ initial data, the argument to globalize the solution is standard. 
 The exponential control of the $H^\sigma$ norm up to arbitrary time $T >0$ follows again looking at the  
 Volterra integral operator 
 and using \eqref{RickBoundE1}. Once we know that the $H^\sigma$ norm does not blow-up in finite time we can also extend the 
 local $H^\sigma$ flow to arbitrary times $T>0$.
 \end{proof}
 
\begin{lemma}\label{InverseFlowStab1}
Let $s \geq 0$,  $T>0$, $R>0$ and $N \in \N \cup \{ \infty \}$. Assume that $\| w^N \|_{L^2}, \| v^N \|_{L^2} \leq R$.
Then for $|t| \lesssim T  $ it holds
\begin{equation}
\label{FlowPolyTruncated}
 \| \Phi^{N}_{t} (v) - \Phi^{N}_{t} (w)\|_{H^s} \lesssim_{R, T} 
  \| v -  w \|_{H^s}\,.
\end{equation}
\end{lemma}
\begin{proof}
The argument easily follows using Lemma \ref{LWPlemmaNLS}, 
the inequalities of Lemma \eqref{prop campo hamiltoniano X FN} and the Duhamel representation of truncated flows. By time reversibility the same bound holds for the inverse flow. This ends the proof. 
\end{proof}

We now prove the convergence of the flow of the truncated vector field to the one of the non-truncated vector field. 
The flow $\Phi_t \:= \Phi_t^{\mathcal F}$ is the flow of the Hamiltonian vector field $X_{\mathcal F}$ whereas the flow of $\Phi_t^N \:= \Phi_t^{{\mathcal F}_N}$ is the flow associated to the Hamiltonian vector field $X_{{\mathcal F}_N}$ (recall \eqref{campo hamiltoniano NLS}) where ${\mathcal F}_N(u) = {\mathcal F}(\Pi_N u)$. 
\begin{lemma}\label{stime differenze campi vettoriali}
Let $u_0 \in L^2(\T)$. It holds
\begin{equation}\label{NonUnifFC}
\sup_{\tau \in [- 1, 1]} \| \Phi_\tau (u_0) - \Phi_\tau^N (\Pi_N u_0) \|_{L^2}  \to 0 \quad \text{as} \quad N \to + \infty\,. 
\end{equation}
Moreover, if $K \subset L^{2}(\T)$ is compact, then 
\begin{equation}\label{UnifFC}
\sup_{u_0 \in K} \sup_{\tau \in [- 1, 1]} \| \Phi_\tau (u_0) - \Phi_\tau^N (\Pi_N u_0) \|_{L^2}  \to 0 \quad \text{as} \quad N \to + \infty\,. 
\end{equation}
\end{lemma}
\begin{proof}
Recalling \eqref{DefTrunc1} and by setting $u_N := \Pi_N u$, one has
\begin{equation}\label{X - XN}
\begin{aligned}
X_{\mathcal F}(u) - X_{{\mathcal F}_N}(u_N) & = X_{\mathcal F}(u) - \Pi_N X_{\mathcal F}(u_N) \\
& \stackrel{{\rm Id} = \Pi_N + \Pi_N^\bot }{=}  \Pi_N^\bot X_{\mathcal F}(u) + \Pi_N X_{\mathcal F}(u) - \Pi_N X_{\mathcal F}(u_N)\\
& =  {\mathcal R'}_N(u) + 
{\mathcal R}_N(u) 
  \quad \text{where} \\
{\mathcal R'}_N(u) &:= \Pi_N X_{\mathcal F}(u) - \Pi_N X_{\mathcal F}(u_N)
\quad {\mathcal R}_N(u) := \Pi_N^\bot X_{\mathcal F}(u)\,. 
\end{aligned}
\end{equation}
The latter computation is needed in order to estimate the difference of the solutions of the following Cauchy problem. Given $U_0 \in L^2(\T)$, we consider 
\begin{equation}\label{prob cauchy e prob troncato}
\begin{cases}
\partial_t u = X_{\mathcal F}(u) \\
u(0) = u_0,
\end{cases} \qquad \begin{cases}
\partial_t u_N = X_{{\mathcal F}_N}(u_N) \\
u_N(0) = \Pi_N u_0,
\end{cases}
\end{equation}
with $\sup_{\t\in[-1,1]}\| u_N \|_{L^2}\leq R$ and $\sup_{\t\in[-1,1]}\| u \|_{L^2} \leq R$, for some $R \geq 0$. By \eqref{X - XN}, one obtains that $\delta_N := u - u_N$ satisfies the following problem: 
\begin{equation}\label{problema U - UN}
\begin{cases}
\partial_\tau \delta_N = {\mathcal R'}_N(u) + {\mathcal R}_N(u) \\
\delta_N(0) = \Pi_N^\bot u_0\,. 
\end{cases}
\end{equation}
This implies that 
$$
\delta_N(\tau) = \Pi_N^\bot u_0 + \int_0^\tau {\mathcal R'}_N(u(z, \cdot))\, d z + \int_0^\tau {\mathcal R}_N(u(z, \cdot))\, d z.
$$
since $\sup_{\t\in[-1,1]}\| u_N \|_{L^2}\leq R$ and $\sup_{\t\in[-1,1]}\| u \|_{L^2} \leq R$, by applying Lemma \ref{prop campo hamiltoniano X FN} item $(ii)$, one obtains that for any $z \in [- 1, 1]$
$$
\| {\mathcal R'}_N(u(z, \cdot)) \|_{L^2} \lesssim 
\Big(  \sup_{\t\in[-1,1]}\| u_N \|_{L^2} + \sup_{\t\in[-1,1]}\| u \|_{L^2} \Big)^2 \| \delta_N \|_{L^2} \lesssim R^2 \| \delta_N \|_{L^2}
$$
implying that $\delta_N(\tau)$ satisfies the integral inequality 
$$
\| \delta_N(\tau) \|_{L^2} \leq \| \Pi_N^\bot u_0 \|_{L^2} + \int_{- 1}^1 \| {\mathcal R}_N(u(z, \cdot)) \|_{L^2}\, d z + C R^2 \Big| \int_0^\tau \| \delta_N(z) \|_{L^2}\, d z \Big| \,.
$$
By the Gronwall inequality, one then gets that 
$$
\sup_{\t\in[-1,1]}\| \delta_N \|_{L^2} = \sup_{\tau \in [- 1, 1]} \| \delta_N(\tau) \|_{L^2} \lesssim_R \| \Pi_N^\bot u_0 \|_{L^2} + \int_{- 1}^1 \| {\mathcal R}_N(u(z, \cdot)) \|_{L^2}\, d z\,.
$$
Clearly 
\begin{equation}\label{DiniHp1}
\Pi_N^\bot u_0 \to 0, 
\quad \text{as $\quad N \to  \infty$ in $L^2(\T)$}
\end{equation} 
in $L^2(\T)$.
We shall show that 
\begin{equation}\label{cappellaio matto}
\int_{- 1}^1  \| {\mathcal R}_N(u(z, \cdot)) \|_{L^2}\, d z \to 0 
\end{equation}
 as $N \to + \infty$. This can be proved by the Lebesgue dominate convergence. Indeed, again by applying Lemma \ref{prop campo hamiltoniano X FN}-$(i)$, one gets that for any $\tau \in [- 1, 1]$
$$
\| X_{\mathcal F}(u(\tau)) \|_{L^2} \lesssim \| u(\tau) \|_{L^2}^3 \lesssim R^3\,.
$$
This implies that for any $\tau \in [- 1, 1]$, 
\begin{equation}\label{DiniHp2}
\Pi_N^\bot X_{{\mathcal F}}(u (\tau)) \to 0 \quad \text{as $\quad N \to  \infty$ in $L^2(\T)$}
\end{equation}
 and 
$$
\sup_{N \in \N} \sup_{\tau \in [- 1, 1]} \| {\mathcal R}_N(u(\tau)) \|_{L^2} \lesssim R^3\,.
$$
Therefore \eqref{cappellaio matto} follows by dominated convergence. 
This completes the proof of \eqref{NonUnifFC}. To show \eqref{UnifFC} we note that the 
sequences \eqref{DiniHp1}-\eqref{DiniHp2} are monotone and that the functions 
$$
(U_0, \tau) \in (K, [-1,1]) \to  \Phi_t(u_0) -  \Phi_t^N(u_0)
$$  
are continuous (by Lemma \ref{InverseFlowStab}), thus the point-wise convergence to zero is promoted to uniform convergence by using the Dini criterion.
\end{proof}

Recall that we are abbreviating $\Phi \:= (\Phi_t)\big|_{t=1}$ and 
$\Phi^N \:= (\Phi_t^N)\big|_{t=1}$

\begin{lemma}\label{lemma:recall0}
Let $A$ be a compact set in the $L^2(\T)$ topology. For all $\varepsilon >0$ we can find $N_{\varepsilon}$
sufficiently large, so that
\begin{equation}\label{FlowStability1}
\Phi(A) \subset 
\Phi^{N} (A +  B(\varepsilon)), \qquad \forall N>N_{\varepsilon}.
\end{equation}
\end{lemma}
\begin{proof}
Set
$$
\Psi_N(v) := (\Phi^N)^{-1} \Phi (v)\,.
$$
Then \eqref{FlowStability}
follows once we prove that for all $v \in A$ and all $\e>0$ there is $N_\e$ large enough such that
\begin{equation}\label{laterBettExpl1}
\| v - \Psi_N (v)\|_{L^{2}(\T)} < \varepsilon\,\quad\text{for all } N>N_\e.
\end{equation}
Indeed, assuming \eqref{laterBettExpl1} we have for any $v\in A$
$$
\Psi_N (v)\in v+B(\e)\subset A+B(\e)
$$
thus
$$
\Phi  (v)=\Phi^{N} \Psi_N (v)\in \Phi^{N}(v+B(\e))\,,
$$
which readily implies \eqref{FlowStability1}. The proof of \eqref{laterBettExpl1}:
\bea
\| v - \Psi_N (v)\|_{L^{2}(\T)}&=& \|(\Phi^N)^{-1} \Phi^N  (v)-(\Phi^N)^{-1} \Phi  (v)\|_{L^{2}(\T)}\nn\\
&\lesssim_R& \|\Phi^N  (v)-\Phi  (v)\|_{L^{2}(\T)}\nn\,,
\eea
where we used Lemma \ref{InverseFlowStab1} and the fact that the $L^2$ norms of 
$\Phi^N  (v)$ and of $\Phi  (v)$ are bounded by $\|v\|_{L^2} = R$. 
By Lemma \ref{stime differenze campi vettoriali} for all $\e>0$ there is a $N_\e$ such that 
$$
\|\Phi^N  (v)-\Phi  (v)\|_{L^{2}(\T)}\leq \e\,,
$$
and \eqref{laterBettExpl} is proved.
\end{proof}


\section{The flow of the Birkhoff map for the fractional NLS}\label{section:flow-fract}

In this section we prove local well-posedness for the flow of the Birkhoff map associated to the fractional NLS, \ie $\a\neq1$. We will do it for $\a>\frac34$, that is sufficient for our purposes. 

Let us shorten
\be\label{eq:defPhi}
\Phi(j_1, j_2, j_3) := j_1^{2 \alpha} + j_2^{2 \alpha} - j_3^{2 \alpha} - n^{2 \alpha} \quad \text{with} \quad  n= j_1 + j_2 - j_3\,. 
\ee
We will crucially use the following lower bound (see for instance \cite[Lemma 2.4]{Forlano})
\be\label{lower bound divisori frazionari}
|\Phi(j_1, j_2, j_3) | \gtrsim |j_1 - j_3| |j_2 - j_3| \big( \langle j_1 \rangle + \langle j_2 \rangle + \langle j_3 \rangle  \big)^{ - (2 - 2 \alpha)}\,. 
\ee

 \begin{lemma}\label{prop campo hamiltoniano X FN frazionario rick}
Let $\alpha > \frac34$ and $N \in \N \cup \{ \infty \}$. Then the following holds.

$(i)$ For any $s \geq 2 - 2 \alpha$, for any $u \in H^s(\T)$, 
\begin{equation}\label{RickBoundE1}
\| X_{{\mathcal F}_N}(u) \|_{H^s} \lesssim_s  \| u \|_{H^s}^2 \| u \|_{L^2}\,. 
\end{equation}
$(ii)$ For any $s \geq 2 - 2 \alpha$, for any $u, v \in H^s (\T)$, one has 
$$
\|  X_{{\mathcal F}_N}(u) - X_{{\mathcal F}_N}(v) \|_{H^s} \lesssim 
\|  u - v \|_{H^s} ( \| u \|_{H^s}^2 + \| v \|_{H^s}^2) .
$$
\end{lemma}
\begin{proof}
The Hamiltonian vector field is $X_{{\mathcal F}_N} = (i \nabla_{\bar u} {\mathcal F}_N, - i \nabla_u {\mathcal F}_N)$ and one computes that 
$$
\begin{aligned}
\partial_{\bar u_n} {\mathcal F}_N & = \sum_{\substack{
|j_1|, |j_2|, |j_3| \leq N \\
j_1 + j_2 - j_3 = n\\
j_1^{2 \alpha} + j_2^{2 \alpha} \neq j_3^{2 \alpha} + n^{2 \alpha}
}} \dfrac{\sigma i }{2  \Phi(j_1, j_2, j_3)}  u(j_1)  u(j_2) \overline u(j_3)\,.
\end{aligned}
$$

Then in order to deduce the desired estimates in $(i)$-$(ii)$, it is enough to prove that the trilinear form ${\mathcal T}$ defined by 
\begin{equation}\label{capolavoro elegante}
\begin{aligned}
& {\mathcal T}[u, v, \varphi] := \sum_{n \in \Z} {\mathcal T}_n[u, v, \varphi] e^{i n x}\,, \\
& {\mathcal T}_n[u, v, \varphi] := \sum_{\substack{
|j_1|, |j_2|, |j_3| \leq N \\
j_1 + j_2 - j_3 = n\\
}} \dfrac{ \sigma i }{2 \Phi(j_1, j_2, j_3)}  u(j_1)  v(j_2) \varphi(j_3) 
\end{aligned}
\end{equation}
is continuous on $H^s$ for $s \geq a = 2 - 2 \alpha$, namely
\begin{equation}\label{prop cal T rick frazionario}
\| {\mathcal T}[u, v, \varphi] \|_{H^s}  \lesssim_s \| u \|_{H^s} \| v \|_{H^s} \| \varphi \|_{L^2} +  \| u \|_{H^s} \| v \|_{L^2} \| \varphi \|_{H^s}  + \| u \|_{L^2} \| v \|_{H^s} \| \varphi \|_{H^s} \,. 
\end{equation}

We have 
\begin{equation}\label{meccanica celeste m - 1}
\begin{aligned}
\| {\mathcal T}[u, v, \varphi] \|_{H^s}^2 & = \sum_{n \in \Z} \langle n \rangle^{2 s} | {\mathcal T}_n[u, v, \varphi]|^2  \\
& \stackrel{\eqref{lower bound divisori frazionari}}{\lesssim}  \sum_{n \in \Z} \Big(  \sum_{\substack{
|j_1|, |j_2|, |j_3| \leq N \\
j_1 + j_2 - j_3 = n\\
}} \dfrac{ \langle n \rangle^\sigma \big( \langle j_1 \rangle + \langle j_2 \rangle + \langle j_3 \rangle  \big)^{a} }{|j_3 - j_2| |j_1 - j_3|}  |u(j_1) | |v(j_2) | |\varphi(j_3) | \Big)^2 \,.
\end{aligned}
\end{equation}
By using that $n = j_1 + j_2 - j_3$, one gets that $\langle n \rangle^s \lesssim_\sigma \langle j_1 \rangle^s + \langle j_2 \rangle^s + \langle j_3 \rangle^s$, implying that 
\begin{equation}\label{meccanica celeste m}
\begin{aligned}
\| {\mathcal T}[u, v, \varphi] \|_{H^\s}^2 & \lesssim_s T_1 + T_2 + T_3\,, \\
& T_1 :=  \sum_{n \in \Z} \Big(  \sum_{\substack{
|j_1|, |j_2|, |j_3| \leq N \\
j_1 + j_2 - j_3 = n\\
}} \dfrac{ \big( \langle j_1 \rangle + \langle j_2 \rangle + \langle j_3 \rangle  \big)^{a}  \langle j_1 \rangle^s }{|j_3 - j_2| |j_1 - j_3|}  |u(j_1) | |v(j_2) | |\varphi(j_3) | \Big)^2  \\
& T_2 :=     \sum_{n \in \Z} \Big(  \sum_{\substack{
|j_1|, |j_2|, |j_3| \leq N \\
j_1 + j_2 - j_3 = n\\
}} \dfrac{ \big( \langle j_1 \rangle + \langle j_2 \rangle + \langle j_3 \rangle  \big)^{a}  \langle j_2 \rangle^s }{|j_3 - j_2| |j_1 - j_3|}  |u(j_1)  | |v(j_2) | |\varphi(j_3) | \Big)^2  \\
& T_3 :=    \sum_{n \in \Z} \Big(  \sum_{\substack{
|j_1|, |j_2|, |j_3| \leq N \\
j_1 + j_2 - j_3 = n\\
}} \dfrac{ \big( \langle j_1 \rangle + \langle j_2 \rangle + \langle j_3 \rangle  \big)^{a}  \langle j_3 \rangle^s }{|j_3 - j_2| |j_1 - j_3|}  |u(j_1) | |v(j_2) |  |\varphi(j_3) | \Big)^2 \,. \\
\end{aligned}
\end{equation}
The estimates of $T_1, T_2, T_3$ can be done similarly, hence we estimate only the term $T_1$. The term $T_1$ can be split into three terms, namely
\begin{equation}\label{splitting T1 frazionario}
\begin{aligned}
T_1 & \lesssim A_1 + A_2 + A_3\,, \\
A_1 & :=  \sum_{n \in \Z} \Big(  \sum_{\substack{
|j_1|, |j_2|, |j_3| \leq N \\
j_1 + j_2 - j_3 = n\\
}} \dfrac{  \langle j_1 \rangle^{s + a}  }{|j_3 - j_2| |j_1 - j_3|}  |u(j_1) | |v(j_2) | |\varphi(j_3) | \Big)^2 \,, \\
A_2 & :=  \sum_{n \in \Z} \Big(  \sum_{\substack{
|j_1|, |j_2|, |j_3| \leq N \\
j_1 + j_2 - j_3 = n\\
}} \dfrac{   \langle j_1 \rangle^s \langle j_2 \rangle^a }{|j_3 - j_2| |j_1 - j_3|}  |u(j_1) | |v(j_2) | |\varphi(j_3) | \Big)^2 \,, \\
A_3 & :=  \sum_{n \in \Z} \Big(  \sum_{\substack{
|j_1|, |j_2|, |j_3| \leq N \\
j_1 + j_2 - j_3 = n\\
}} \dfrac{    \langle j_1 \rangle^s \langle j_3 \rangle^{a} }{|j_3 - j_2| |j_1 - j_3|}  |u(j_1) | |v(j_2) | |\varphi(j_3) | \Big)^2 \,. 
\end{aligned}
\end{equation}
The terms $A_1, A_2, A_3$ are estimated by similar techniques, hence we only provide an estimate of $A_1$. By the triangular inequality, one has 
$$
\begin{aligned}
\langle j_1 \rangle^{s + a} & = \langle j_1 \rangle^s \langle j_1 \rangle^a \lesssim \langle j_1 \rangle^s (\langle j_3 \rangle^a + \langle j_1 - j_3 \rangle^a) \\
& \lesssim  \langle j_1 \rangle^s \langle j_3 \rangle^a  \langle j_1 - j_3 \rangle^a \,.
\end{aligned}
$$
The latter inequality, together with the Cauchy Schwartz inequality implies that for any $n \in \Z$
\begin{equation}\label{meccanica celeste m 2}
\begin{aligned}
& \sum_{\substack{
j_1, j_2 \in \Z
}} \dfrac{ 1 }{\langle j_1 - n \rangle \langle j_2 - n \rangle^{1 - a}} \langle j_1 \rangle^s |u(j_1) |  |v(j_2) | \langle j_1 + j_2 - n \rangle^a |\varphi(j_1 + j_2 - n) |  \\
& \quad \lesssim \Big( \sum_{\substack{
j_1, j_2 \in \Z
}} \langle j_1 \rangle^{2 s} |u(j_1) |^2 |v(j_2) |^2 \langle j_1 + j_2 - n \rangle^{2 a} |\varphi(j_1 + j_2 - n) |^2 \Big)^{\frac12} \Big( \sum_{\substack{
j_1, j_2 \in \Z
}} \dfrac{ 1 }{\langle j_1 - n \rangle^2 \langle j_2 - n \rangle^{2(1 - a)}}  \Big)^{\frac12} \\
& \quad \lesssim  \Big( \sum_{\substack{
j_1, j_2 \in \Z
}} \langle j_1 \rangle^{2 s} |u(j_1) |^2 |v(j_2) |^2 \langle j_1 + j_2 - n \rangle^{2 a} |\varphi(j_1 + j_2 - n) |^2 \Big)^{\frac12} \Big( \sum_{\substack{
k_1, k_2 \in \Z
}} \dfrac{ 1 }{\langle k_1 \rangle^2 \langle  k_2 \rangle^{2(1 - a)}}  \Big)^{\frac12}  \\
& \quad \lesssim \Big( \sum_{\substack{
j_1, j_2 \in \Z
}} \langle j_1 \rangle^{2 s} |u(j_1) |^2 |v(j_2) |^2 \langle j_1 + j_2 - n \rangle^{2 a} |\varphi(j_1 + j_2 - n) |^2 \Big)^{\frac12}
 \end{aligned}
\end{equation}
by using that $2(1 - a) > 1$. 
Thus, $A_1$ can be estimated as 
$$
\begin{aligned}
A_1 & \lesssim \sum_{n \in \Z}   \sum_{\substack{
j_1, j_2 \in \Z
}} \langle j_1 \rangle^{2 s} |u(j_1) |^2 |v(j_2) |^2 \langle j_1 + j_2 - n \rangle^{2 a} |\varphi(j_1 + j_2 - n) |^2  \\
& \lesssim \sum_{j_1 \in \Z} \langle j_1 \rangle^{2 s} |u(j_1)|^2 \sum_{j_2 \in \Z} |v(j_2)|^2 \sum_{n \in \Z} \langle j_1 + j_2 - n \rangle^{2 a} |\varphi(j_1 + j_2 - n) |^2 \\
& \stackrel{j_1 + j_2 - n = k}{\lesssim} \sum_{j_1 \in \Z}  \langle j_1 \rangle^{2 s} |u(j_1)|^2 \sum_{j_2 \in \Z} |v(j_2)|^2 \sum_{k \in \Z}\langle k \rangle^{2 a} |\varphi(k) |^2 \\
& \lesssim \| u \|_{H^s}^2 \| v \|_{L^2}^2 \| \varphi \|_{a}^2 \stackrel{s \geq a}{\lesssim} \| u \|_{H^s} \| v \|_{L^2} \| \varphi \|_{H^s} \,.
\end{aligned}
$$
By similar arguments, one can estimate $A_2, A_3, T_2, T_3$ and hence  
one obtains the bound \eqref{prop cal T rick frazionario}. 
Hence the items $(i)$ and $(ii)$ follow since $\nabla_{\bar u} {\mathcal F}_N = {\mathcal T}[u, u, \overline u]$ and thus
$$\nabla_{\bar u} {\mathcal F}_N(u) - \nabla_{\bar u} {\mathcal F}_N(v) = {\mathcal T}[u, u, \overline u] -
{\mathcal T}[v, v, \overline v] = 
{\mathcal T}[u-v, u, \overline u] + {\mathcal T}[v, u-v, \overline u] + {\mathcal T}[v, v, \overline{u-v} ]$$
\end{proof}
%
%

The latter lemma implies that the flow $\Phi^N_t \:= \Phi^{{\mathcal F}_N}_t$ is well defined in $H^s(\T)$ for any 
$s  \geq  2 - 2 \alpha$ when $\alpha >3/4$. This follows by a standard Picard iteration. More precisely, the following lemma holds: 
\begin{lemma}\label{LWPlemma}
Let $\alpha >3/4$ and $N \in \N \cup \{ \infty \}$. For any $u_0 \in H^s(\T)$, $s \geq 2 - 2 \alpha$, $\| u_0 \|_{H^s} \leq R$, there exists a time $T_\alpha \:= T_\alpha(R) = \frac{\mathtt c_\alpha}{R^2} > 0$, $\mathtt c_\alpha \ll 1$ and a  unique local solution $u \in C^1([- T_\alpha, T_\alpha], H^s(\T))$ which solves the Cauchy problem 
$$
\begin{cases}
\partial_t u(t) = X_{{\mathcal F}_N}(u(t)) \\
u(0) = u_0 .
\end{cases}
$$
 Thus, the flow $\Phi^N_t: H^s(\T) \to H^s(\T)$, $t \in [- T_\alpha, T_\alpha]$
 is a well defined $C^1$ map and 
 $\| u(t) \|_{H^s} \lesssim  \| u_0 \|_{H^s}$. In the case where $\alpha = 1$, one has that $T_\alpha = 1$. 
 \end{lemma}

 \begin{proof}
 The local existence follows by a standard fixed point argument on the Volterra integral operator 
 $$
 u(t) \mapsto {\mathcal S}(u)(t) := u_0 + \int_0^t X_{{\mathcal F}_N}(u(\tau))\, d \tau 
 $$
 in the closed ball 
 $$
 \Big\{ u \in C^0([- T_\alpha, T_\alpha], H^s(\T)) : \sup_{|t|<T_\alpha}\| u \|_{H^s} 
 := \sup_{t \in [- T_\alpha, T_\alpha]}\| u(t) \|_{H^s} \leq R \Big\}\,. 
 $$
 This fixed point argument requires that $R$ is larger than $\| u_0 \|_{H^s}$ and $T R^2 = \mathtt c_\alpha$ with $\mathtt c_\alpha \ll 1$ small enough. If $u$ is a fixed point for ${\mathcal S}$ one also immediately gets that $u \in C^1([- T_\alpha, T_\alpha], H^s(\T))$.  
 \end{proof}
 
\begin{lemma}\label{InverseFlowStab}
Let $\alpha >3/4$, $s \geq  2 - 2 \alpha$ and $N \in \N \cup \{ \infty \}$. Assume that $\| w \|_{H^s}, \| v \|_{H^s} \leq R$.
It is
\begin{equation}
\label{FlowPolyTruncated}
\sup_{|t| < T_\alpha} \| \Phi^{N}_{t} (v) - \Phi^{N}_{t} (w)\|_{H^s} \lesssim_{R, T_{\alpha}, s} 
  \| v -  w \|_{H^s}\,. 
\end{equation}
\end{lemma}
\begin{proof}
The argument easily follows using Lemma \ref{LWPlemma}, 
the inequalities of Lemma \eqref{prop campo hamiltoniano X FN frazionario rick} and the Duhamel representation of truncated flows.
\end{proof}

We now prove the convergence of the flow of the truncated vector field, to the one of the non-truncated vector field. 
The flow $\Phi_t \:= \Phi_t^{\mathcal F}$ is the flow of the Hamiltonian vector field $X_{\mathcal F}$ whereas the flow of $\Phi_N^t \:= \Phi_{{\mathcal F}_N}^t$ is the flow associated to the Hamiltonian vector field $X_{{\mathcal F}_N}$ where ${\mathcal F}_N(u) = {\mathcal F}(\Pi_N u)$. 

\begin{lemma}\label{stime differenze campi vettoriali sigma}
Let $\alpha > 3/4$, $ 2 - 2 \alpha  \leq s' < s $, $R > 1$ and $N \in \N \cup \{ \infty \}$. Then
\begin{equation}\label{GlobalApproxN}
\sup_{\| u_0 \|_{H^s} \leq R} \sup_{\tau \in [- T_\alpha, T_\alpha]}
 \| \Phi_\tau(u_0) -  \Phi_\tau^N (u_0) \|_{H^{s'}} \lesssim R N^{- (s - s')}
\end{equation}
\end{lemma}
\begin{proof}
  Given $u_0 \in H^s(\T)$, we consider 
\begin{equation}\label{prob cauchy e prob troncato a}
\begin{cases}
\partial_t u = X_{\mathcal F}(u) \\
u(0) = u_0,
\end{cases} \qquad \begin{cases}
\partial_t u_N = X_{{\mathcal F}_N}(u_N) \\
u_N(0) = \Pi_N u_0.
\end{cases}
\end{equation}
If $\| u_0 \|_{H^s} \leq R$ then the corresponding solutions satisfy $$\sup_{\tau \in [- T_\alpha, T_\alpha]}\| u_N(\t) \|_{H^s}\,,\, \sup_{\tau \in [- T_\alpha, T_\alpha]}\| u(\t) \|_{H^s} \lesssim R\,.$$ 
By \eqref{X - XN}, one obtains that $v_N := u - u_N$ satisfies the following problem: 
\begin{equation}\label{problema U - UN a}
\begin{cases}
\partial_\tau v_N = {\mathcal A}(u, u_N) v_N + {\mathcal R}_N(u) \\
v_N(0) = \Pi_N^\bot u_0\,. 
\end{cases}
\end{equation}
This implies that 
$$
v_N(\tau) = \Pi_N^\bot u_0 + \int_0^\tau {\mathcal R}_N(u)(z)\, d z + \int_0^\tau {\mathcal A}(u, u_N)(z) v_N (z)\, d z\,.
$$
We have that $\| u \|_{L^\infty_\tau H^s_x}, \| u_N \|_{L^\infty_\tau H^s_x} \leq R$ and we need to estimate $\| v_N(t) \|_{H^{s'}}$ with $s' < s$. By applying Lemma \ref{prop campo hamiltoniano X FN frazionario rick} item $(ii)$, one obtains that for any $z \in [- T_\alpha, T_\alpha]$
$$
\| {\mathcal A}(u, u_N)(z)[\varphi] \|_{H^{s'}} \lesssim \Big(  \sup_{\tau \in [- T_\alpha, T_\alpha]}\| u(\t) \|_{H^{s'}} + \sup_{\tau \in [- T_\alpha, T_\alpha]}\| u_N(\t) \|_{H^{s'}} \Big)^2 \| \varphi \|_{H^{s'}} \lesssim R^2 \| \varphi \|_{H^{s'}}
$$
implying that $V_N(\tau)$ satisfies the integral inequality 
\begin{equation}\label{mfdklsdkngkrls}
\| v_N(\tau) \|_{H^{s'}} \leq \| \Pi_N^\bot u_0 \|_{H^{s'}} + \int_{- T_\alpha}^{T_\alpha} \| {\mathcal R}_N(u)(z) \|_{H^{s'}}\, d z + C R^2 \Big| \int_0^\tau \| v_N(z) \|_{H^{s'}}\, d z \Big| \,, \quad \tau \in [- T_\alpha, T_\alpha]\,. 
\end{equation}
We have (recall $T_{\alpha} = \frac{\mathtt c_\alpha}{R^2}$)
$$
\Big| \int_0^\tau \| v_N(z) \|_{H^{s'}}\, d z \Big| 
\leq  \frac{\mathtt c_\alpha}{R^2} \sup_{|\tau| \leq T_{\alpha}} \| v_N(z) \|_{H^{s'}}
$$
Plugging this into inequality \eqref{mfdklsdkngkrls} and then taking the
 sup over $\tau \in [- T_\alpha, T_\alpha]$ of the new inequality 
 we arrive to
$$
\sup_{|\tau| \leq T_{\alpha}} \| v_N(\tau) \|_{H^{s'}} \leq \| \Pi_N^\bot u_0 \|_{H^{s'}} + \int_{- T_\alpha}^{T_\alpha} \| {\mathcal R}_N(u)(z) \|_{H^{s'}}\, d z + C c_\alpha
\sup_{|\tau| \leq T_{\alpha}} \| v_N(z) \|_{H^{s'}} \,. 
$$
Since $c_\alpha \ll1$ we can reabsorb the last term on the right hand side into the left hand side and we arrive to
$$
\sup_{|\tau| \leq T_{\alpha}} \| v_N(\tau) \|_{H^{s'}} \lesssim \| \Pi_N^\bot u_0 \|_{H^{s'}} + \int_{- T_\alpha}^{T_\alpha} \| {\mathcal R}_N(u)(z) \|_{H^{s'}}\, d z  
$$
By standard smoothing properties, one has that 
$$
\| \Pi_N^\bot u_0 \|_{H^{s'}} \lesssim N^{- (s - s')} \| u_0 \|_{H^s} \lesssim R N^{- (s - s')} 
$$
and by using also Lemma \ref{prop campo hamiltoniano X FN frazionario rick}-$(i)$, one gets that 
$$
\begin{aligned}
\int_{- T_\alpha}^{T_\alpha} \| {\mathcal R}_N(u)(z) \|_{H^{s'}}\, d z & =  \int_{- T_\alpha}^{T_\alpha} \| \Pi_N^\bot X_{\mathcal F}(u(z)) \|_{H^{s'}}\, d z \lesssim N^{- (s - s')} \int_{- T_\alpha}^{T_\alpha} \| X_{\mathcal F}(u(z)) \|_{H^s}\, d z  \\
& \lesssim N^{- (s - s')} \int_{- T_\alpha}^{T_\alpha} \| u(z) \|_{H^s}^3\, d z \lesssim N^{- (s - s')} 
\frac{\mathtt c_\alpha}{R^2} R^3 \lesssim  c_\alpha R N^{- (s - s')}\,. 
\end{aligned}
$$
This implies
$$
\sup_{|\tau| \leq T_{\alpha}} \| v_N(\tau) \|_{H^{s'}} \lesssim R N^{- (s - s')}
$$
from which we deduce the statement.
\end{proof}

We need now an approximation result that allows to construct a flow $\Phi_t$ on $t \in [-1,1]$ 
once we have suitable 
estimates on the approximated flow $\Phi_t^N$ that are uniform over $N \in \N$.

\begin{proposition}\label{ApproxThm}
Let $\alpha >3/4$ and $2 - 2 \alpha  \leq s' < s $, $R>0$ and $\varepsilon > 0$. 
Let $A \subset B_s(R)$. 
There exists $N$ sufficiently large (depending on $\alpha, s, s',\varepsilon, R$) such that the following holds. If
\begin{equation}\label{GrowthAssumption}
\sup_{t \in [-1, 1]} \sup_{u_0 \in A} \| \Phi^N_{t} (u_0) \|_{H^{s}} \leq R , 
\end{equation}
then
the flow $\Phi_t(u_0)$ is well defined on $t \in [-1, 1]$ for all $u_0 \in A$. Moreover
\begin{equation}\label{approxProp} 
\sup_{t \in [-1, 1]} \| \Phi_t(u_0) - \Phi^N_{t} (u_0) \|_{H^{s'}} \leq \varepsilon, 
\quad \forall u_0 \in A\,.
\end{equation}
\end{proposition}
\begin{remark}\label{remark flusso troncato}
The assumption \eqref{GrowthAssumption} on the truncated flow $\Phi_t^N$ will be verified in Proposition \ref{prop:alabourgain}. 
\end{remark}
\begin{proof}
Recall that $\mathtt{c}_{\alpha}=T_\a R\ll1$, where $T_\a$ is the local existence time of $\Phi_t$ (see Lemma \ref{LWPlemma}).
Let $J$ be the smallest integer such that $J \frac{\mathtt{c}_{\alpha} }{2(R^2+1)}  \geq 1$.
We have $$\frac{2(R^2+1)}{\mathtt{c}_{\alpha}}  \leq J < \frac{2(R^2+1)}{\mathtt{c}_{\alpha}} + 1.$$  
We partition the interval $[0,1]$ into $J-1$ intervals of length $\frac{\mathtt{c}_{\alpha}}{2(R^2+1)}$ and a 
last, possibly smaller interval. 
We will compare the approximated flow $\Phi^N_{t}$ 
and $\Phi_{t}$ (which exists only for small times) on these small intervals and we will glue the local solutions.
Let  
$$s'  < s_{J} < \ldots <  s_2 < s_1 < s_0=s\,.$$
We proceed by induction over $j = 0, \dots, J$. Assuming that  $\Phi_{t}$ is well defined 
on $\big[0, (j+1)\frac{\mathtt{c}_{\alpha}}{2(R^2+1)} \big]$  and that
\begin{equation}\label{Growth1}
\sup_{t \in \big[0, j\frac{\mathtt{c}_{\alpha}}{2(R^2+1)} \big]} 
\|  \Phi_t(u_0) -  \Phi_t^N(u_0)   \|_{H^{s_{j}}} \leq  N^{-\kappa_j} \, ,
\end{equation}
for some $\kappa_j >0$, we will show that 
$\Phi_t(u_0)$ is well defined on $\big[0, (j+2)\frac{\mathtt{c}_{\alpha}}{2(R^2+1)}\big]$  and that
\begin{equation}\label{Growth2}
\sup_{t \in \big[0, (j+1) \frac{\mathtt{c}_{\alpha}}{2(R^2+1)}\big]} 
\|  \Phi_t(u_0) -  \Phi_t^N(u_0) \|_{H^{s_{j+1}}} \leq   N^{-\kappa_{j+1}} \, ,
\end{equation}
for a suitable $\k_{j+1} >0$, 
provided $N$ is sufficiently large. In particular, we take $N$ so large in such away that we also have 
$N^{-\kappa_{j+1}} < \varepsilon$. 
Using the induction procedure up to~$j = J$, the statement would then follows.

The induction base $j=0$ is covered by Lemma \ref{LWPlemma} and Lemma \ref{stime differenze campi vettoriali sigma} and by 
the fact that $A \subset B_{s}(R)$.

Regarding the induction step, we first prove \eqref{Growth2}. Then using the
 assumption \eqref{GrowthAssumption} and the triangle 
inequality we get
\begin{equation}\label{EqBFinStable}
\sup_{t \in \left[0, (j+1) \frac{\mathtt{c}_{\alpha}}{2(R^2+1)}\right]} 
\|  \Phi_t(u_0) \|_{H^{s_{j+1}}} \leq  R +  N^{-\kappa_{j+1}} < 2 R.
\end{equation}
By \eqref{EqBFinStable} we then use Lemma \ref{LWPlemma} (with $2R$ in place of $R$) to show that
$\Phi_t(u_0)$ is well defined on $\big[0, (j+2)\frac{\mathtt{c}_{\alpha}}{2(R^2+1)}\big]$.

Now we show \eqref{Growth2}. If the $\sup$ in \eqref{Growth2} is attained
for~$t \in \big[0, j \frac{\mathtt{c}_{\alpha}}{2(R^2+1)} \big]$, then \eqref{Growth2} follows 
by \eqref{Growth1} simply taking $\kappa_{j+1} = \kappa_{j}$.
On the other hand, if the $\sup$ is attained 
for $t \in \big[ j \frac{\mathtt{c}_{\alpha}}{2(R^2+1)} , (j+1) \frac{\mathtt{c}_{\alpha}}{2(R^2+1)} \big]$, using the 
group property of the flow, we need to prove 
\begin{equation}\label{Growth3}
\sup_{t \in \big[j \frac{\mathtt{c}_{\alpha}}{2(R^2+1)},  (j+1)\frac{\mathtt{c}_{\alpha}}{2(R^2+1)} \big]} 
\|  \Phi_{t}  \Phi_{ j \frac{\mathtt{c}_{\alpha}}{2(R^2+1)} } (u_0) -  \Phi_{t}^N   \Phi_{ j \frac{\mathtt{c}_{\alpha}}{2(R^2+1)}}^N (u_0) \|_{H^{s_{j+1}}} 
\leq   N^{-\kappa_{j+1}} \, .
\end{equation}
To do so we decompose 
\begin{align}\nonumber
\|  \Phi_{t}  \Phi_{j \frac{\mathtt{c}_{\alpha}}{2(R^2+1)} } (u_0) & 
-  \Phi_{t}^N  \Phi_{j \frac{\mathtt{c}_{\alpha}}{2(R^2+1)}}^N  (u_0) \|_{H^{s_{j+1}}}
\\ \label{Term1}
& 
\leq \|  \Phi_{t}  \Phi_{j \frac{\mathtt{c}_{\alpha}}{2(R^2+1)}} (u_0) -  \Phi_{t}  \Phi_{j \frac{\mathtt{c}_{\alpha}}{2(R^2+1)}}^N (u_0) \|_{H^{s_{j+1}}}
\\ \label{Term2}
& + \|  \Phi_{t}  \Phi_{j \frac{\mathtt{c}_{\alpha}}{2(R^2+1)}}^N (u_0) -  \Phi_{t}^N  \Phi_{j \frac{\mathtt{c}_{\alpha}}{2(R^2+1)}}^N  (u_0) \|_{H^{s_{j+1}}}
\end{align}
and we will handle these two terms separately.

To bound \eqref{Term1} we first note that by the induction assumption \eqref{Growth1} and by the assumption \eqref{GrowthAssumption} we have
for $N$ large enough
$$ 
\|  \Phi_{j \frac{\mathtt{c}_{\alpha}}{2(R^2+1)}} (u_0) \|_{H^{s_j}} \leq R + N^{-\kappa_{j}} < R+\varepsilon.
$$
Using this fact, the assumption \eqref{GrowthAssumption} and the fact the stability estimate \eqref{FlowPolyTruncated}  
is time-translation invariant, we apply \eqref{FlowPolyTruncated} to get
\begin{align}\label{eq:rhs-of}
& \sup_{t \in \big[j \frac{\mathtt{c}_{\alpha}}{2(R^2+1)},  (j+1)\frac{\mathtt{c}_{\alpha}}{2(R^2+1)} \big] }
\|  \Phi_{t}  \Phi_{j\frac{\mathtt{c}_{\alpha}}{2(R^2+1)} }(u_0) -  \Phi_{t}  \Phi_{j \frac{\mathtt{c}_{\alpha}}{2(R^2+1)} }^N (u_0) \|_{H^{s_{j+1}}}
\\ &
\lesssim_{R + \varepsilon, s}    
\|   \Phi_{j \frac{\mathtt{c}_{\alpha}}{2(R^2+1)} } (u_0) -   \Phi_{j \frac{\mathtt{c}_{\alpha}}{2(R^2+1)} }^N (u_0) \|_{H^{s_{j+1}}} \leq 
N^{-\kappa_{j+1}}
\,,
\end{align}
where in the last step we used the induction assumption \eqref{Growth1}, where $0>\kappa_{j+1}> s_{j+1} - s_j$ and $N$ sufficiently large. 

To bound the term \eqref{Term2}
we use the stability estimate \eqref{stime differenze campi vettoriali} 
and initial datum 
$\Phi_{j \frac{\mathtt{c}_{\alpha}}{2(R^2+1)} }^N (u_0)$,
that is allowed recalling the assumption \eqref{GrowthAssumption}. Thus for $\kappa_{j+1}$
as above we arrive to 
$$
\sup_{t \in [ j \frac{\mathtt{c}_{\alpha}}{2(R^2+1)},  (j+1)\frac{\mathtt{c}_{\alpha}}{2(R^2+1)} ]} 
\|  \Phi_{t}  \Phi_{j \frac{\mathtt{c}_{\alpha}}{2(R^2+1)}}^N (u_0)
-  \Phi_{t}^N  \Phi_{j \frac{\mathtt{c}_{\alpha}}{2(R^2+1)} }^N (u_0) \|_{H^{s_{j+1}}}
\lesssim_R  N^{s_{j+1}- s_j} < \frac12 N^{-\kappa_{j+1}},
$$
provided that $N$ is sufficiently large. This concludes the proof.
\end{proof}

Recall that we are abbreviating $\Phi \:= (\Phi_t)\big|_{t=1}$ and $\Phi^N \:= (\Phi_t^N)\big|_{t=1}$. 

\begin{lemma}\label{lemma conditional stability}
Let $\alpha > 3/4$ and $ s \geq 2- 2\alpha$, $R>0$ and $\varepsilon > 0$. 
Let $A \subset {B}_s(R)$. 
Assume that 
\begin{equation}\label{GrowthAssumptionB1}
\sup_{N \in \N}\sup_{t \in [-1, 1]} \sup_{u_0 \in A} \| \Phi^N_{t} (u_0) \|_{H^s} \leq R. 
\end{equation}
Then if $u_0, v_0 \in A$, then 
$$
\| \Phi^N_{t} (u_0) - \Phi^N_{t} (v_0)  \|_{H^s} \lesssim_R \| u_0 - v_0 \|_{H^s}, \qquad t \in [-1, 1].
$$
\end{lemma}
\begin{proof}
By the Duhamel representation of the solution one has
$$
\Phi^N_{t} (u_0)  - \Phi^N_{t} (v_0) = u_0 - v_0 + \int_{0}^{t} \delta_N(\tau) \, d \tau, \quad \delta_N(\tau) :=  X_{{\mathcal F}_N}(\Phi^N_{\tau} (u_0)) - X_{{\mathcal F}_N}(\Phi^N_{\tau} (v_0) ), \quad \tau \in [- 1, 1]\,. 
$$ 
By the estimates of Lemma \ref{prop campo hamiltoniano X FN frazionario rick} and by using the assumption \eqref{GrowthAssumptionB1}, we get 
$$
\| \delta_N(\tau) \|_{H^s} \lesssim_s R^2 \| \Phi_\tau^N (u_0) - \Phi_\tau^N (v_0) \|_{H^s}, \quad \forall \tau \in [- 1, 1],
$$
and hence 
$$
\| \Phi^N_{t} (u_0) - \Phi^N_{t} (v_0) \|_{H^s} \leq \| u_0 - v_0 \|_{H^s} + C R^2 \Big| \int_{0}^{t} \| \Phi_\tau^N (u_0) - \Phi_\tau^N (v_0) \|_{H^s} \, d \tau \Big|\,.
$$
This implies the claimed bound by using the Gr\"onwall inequality. 

\end{proof}

\begin{lemma}\label{lemma:recall}
Let $\alpha > 3/4$, $2 - 2 \alpha \leq s' < s $, $R>0$ and $\varepsilon > 0$. 
Let $A \subset {B}_s(R)$. 
Assume that 
\begin{equation}\label{GrowthAssumptionB}
\sup_{N \in N} \sup_{t \in [-1, 1]} \sup_{u_0 \in A} \| \Phi^N_{t} (u_0) \|_{H^{s}} \leq R , 
\end{equation}

Let $E \subset A$ be a compact set in the $H^s(\T)$ topology. Then, for 
all $\varepsilon >0$, $2 - 2 \alpha  \leq s' < s$ there exists $N_{\varepsilon}$
sufficiently large, so that
\begin{equation}\label{FlowStability}
\Phi(E) \subset 
\Phi^{N} (E +  B_{s'}(\varepsilon)), \qquad \forall N>N_{\varepsilon '}.
\end{equation}
\end{lemma}
\begin{proof}
Set
$$
\Psi_N v := (\Phi^N)^{-1} \Phi (v)\,.
$$
Then \eqref{FlowStability}
follows once we prove for all $v \in E$ and all $\e>0$ there is $N_\e$ large enough such that
\begin{equation}\label{laterBettExpl}
\| v - \Psi_N (v)\|_{H^{s'}} < \varepsilon\,\quad\text{for all } N>N_\e.
\end{equation}
Indeed, assuming \eqref{laterBettExpl} we have for any $v\in E$
$$
\Psi_N (v)\in v+B_{s'}(\e)\subset E +B_{s'}(\e)
$$
thus
$$
\Phi (v) = \Phi^{N} \Psi_N (v) \in \Phi^{N}(v+B_{s'}(\e))\,,
$$
which readily implies \eqref{FlowStability}. The proof of \eqref{laterBettExpl}:
\bea
\| v - \Psi_N (v)\|_{H^{s'}}&=& \|(\Phi^N)^{-1} \Phi^N  (v) - (\Phi^N)^{-1} \Phi  (v) \|_{H^{s'}}\nn\\
&\stackrel{\text{by Lemma \ref{lemma conditional stability}}}{\lesssim_R}& \|\Phi^N  (v) - \Phi  (v)\|_{H^{s'}}\nn\,,
\eea
 Finally, by Proposition \ref{ApproxThm} for all $\e>0$ there is a $N_\e$ such that 
$$
\|\Phi^N (v) - \Phi  (v) \|_{H^{s'}}\leq \e\,,
$$
 \eqref{laterBettExpl} is proved and the proof is concluded. 
\end{proof}


\section{Main probabilistic estimates}\label{section:prob}

The main result of this section is the following. 

\begin{proposition}\label{prop:Lpbound}
Let $1 \geq \alpha >  \frac{1+ \sqrt{97}}{12} \sim 0.9$ and set
\be\label{eq:zeta}
\zeta(\a):=\min\left(\frac13+\frac{2\a-1}{3(1-\a)}, \frac{2 \a (\a +1)}{4 - 4\a^2 + 3\a} \right)>1\,.
\ee
For all $p\geq1$ it holds
\be\label{eq:Lpestimate}
\left\|\frac{d}{dt} H[\Phi^N_t(u)]\big|_{t=0}\right\|_{L^p(\tilde\g_{\alpha})}\lesssim_R p^{\frac{1}{\zeta(\a)}}\,.
\ee
\end{proposition}
\begin{remark}
Note that $\zeta(1) = 4/3$. 
\end{remark}
We start with the following deterministic estimate.
\begin{lemma}\label{lemma:dtE1}
For all $\a\in(3/4,1]$ it holds
\be
\left|\frac{d}{dt} H^{(\a)}[\Phi^N_t (u)]\big|_{t=0}\right|\lesssim  \|\Pi_Nu\|^4_{L^2}+ 
\|\Pi_Nu\|^2_{L^2}\|\Pi_Nu\|_{FL^{0,1}}^2+\|\Pi_Nu\|_{L^2} \|\Pi_Nu\|^2_{H^{2-2\alpha}} \|\Pi_Nu\|^3_{FL^{0,1}}\,.
\ee
\end{lemma}
\begin{proof}
We have
\bea\label{MainIdentity}
\frac{d}{dt} H^{(\a)}[\Phi^N_t (u)]\big|_{t=0}&=&\{H^{(\a)}[\Pi_N u],\mathcal F_N\}\nn\\
&=& \{ \||D_x|^{\alpha} \Pi_Nu\|^2_{L^2}, {\mathcal F}_N \} + \frac{\sigma}2 \{\|\Pi_Nu\|^4_{L^4},{\mathcal F}_N\}\nn\\
&\stackrel{\eqref{omologica cal FN}}{=}&- \frac{\sigma}{2}\|\Pi_Nu\|^4_{L^4}+ \sigma \|\Pi_Nu\|^4_{L^2}+ \frac{\sigma}2 \{\|\Pi_Nu\|^4_{L^4},{\mathcal F}_N\}\,. \label{eq:outdt}
\eea
The first summand in (\ref{eq:outdt}) is bounded by
\be\label{eq:rci0}
\|\Pi_Nu\|^4_{L^4}\leq \|\Pi_Nu\|^2_{L^2}\left(\sum_{|n|\leq N}|u(n)|\right)^2 \leq  \|\Pi_Nu\|^2_{L^2}  \|\Pi_Nu\|^2_{FL^{0, 1}}\,.
\ee
Hence, we need to estimate only the term $\{\|\Pi_Nu\|^4_{L^4},{\mathcal F}_N\}$. By the definition of the Poisson brackets given in \eqref{def Poisson}, one has that 
$$
\begin{aligned}
& \{\|\Pi_Nu\|^4_{L^4},{\mathcal F}_N(u)\} = D_U {\mathcal G}(U)[X_{{\mathcal F}_N}(U)] \quad \text{where} \\
& {\mathcal G}(U) := \int_\T u^2 \bar u^2\, d x\,. 
\end{aligned}
$$
We write
$$
\begin{aligned}
& D_U {\mathcal G}(U) [h] = {\mathcal T}_1 h + {\mathcal T}_2 \bar h\,, \\
& {\mathcal T}_1 h :=  2 \int_\T u \bar u^2 h\, d x, \quad {\mathcal T}_2 \bar h :=  2 \int_\T u^2 \bar u \bar h\,dx \,. 
\end{aligned}
$$
We have
\be\label{eq"rci1}
\begin{aligned}
|{\mathcal T}_1 h | & \lesssim \sum_{j_1 + j_2 - j_3 - j_4 = 0} |h(j_1)| |u(j_2)| |u(j_3 )| | u(j_4)| \\
& \lesssim \sum_{j_2, j_3, j_4} |h(j_3 + j_4 - j_2)| |u(j_2)|  |u(j_3 )| | u(j_4)|  \\
& \lesssim \| h \|_{L^2} \sum_{j_2, j_3, j_4}  |u(j_2)|  |u(j_3 )| | u(j_4)| \lesssim \| h \|_{L^2} \| u \|_{FL^{0, 1}}^3 \\
& \text{and similarly} \quad |{\mathcal T}_2 \bar h| \lesssim \| h \|_{L^2} \| u \|_{FL^{0, 1}}^3 \,. 
\end{aligned}
\ee
Therefore 
\be\label{eq"rci2}
|D_U {\mathcal G}(U)[X_{\mathcal F_N}(U)]| \lesssim \| u \|_{FL^{0, 1}}^3 \| X_{\mathcal F_N}(U) \|_{L^2}\,.
\ee
Combining \eqref{eq"rci2} and \eqref{RickBoundE1} (recall that $2 - 2 \a \geq 0$) gives
\be\label{eq"rci3}
\begin{aligned}
|\{\|\Pi_Nu\|^4_{L^4},{\mathcal F}_N(u)\}| & = |D_U {\mathcal G}(U)[X_{{\mathcal F}_N}(U)]|  \lesssim \| u \|_{FL^{0, 1}}^3 \| X_{{\mathcal F}_N}(U) \|_{L^2} \\
& \lesssim \| u \|_{FL^{0, 1}}^3 \| X_{{\mathcal F}_N}(U) \|_{2 - 2 \a} \lesssim \| u \|_{FL^{0, 1}}^3 \| u \|_{2 - 2 \a}^2 \| u \|_{L^2}\,. 
\end{aligned}
\ee
Using \eqref{eq:rci0} and \eqref{eq"rci3} in \eqref{MainIdentity} we finish the proof. 
\end{proof}

We now establish the following tail estimates.

\begin{lemma}\label{lemma:subexpHs}
Let $s<\a-\frac12$. There is $c(R)>0$ such that for all $t>0$
\be\label{eq:subexpFL}
\tilde\g_{\alpha,N}\left(\|\Pi_Nu\|_{H^{s}}\geq t\right)\leq \exp\left(-c(R)t^{\frac{2\a}{s}}\right) \,.
\ee
\end{lemma}
\begin{proof}
Since
\begin{equation*}
\|\Pi_N u\|_{H^s}
\lesssim\sum_{j\in\N}2^{js}\|\D_j \Pi_N u\|_{L^2}\,,
\end{equation*}
we have
\be\label{eq:11}
\tilde\g_{\a,N}(\|\Pi_Nu\|_{H^s}\geq t)\lesssim \tilde\g_{\a,N}(\sum_{j\in\N}2^{js}\|\D_j \Pi_Nu\|_{H^s}\geq t)\,. 
\ee

Let now 
\begin{equation}\label{Rec1pp}
j_t:=\min\{j\in\N\,:\, 2^j\geq (t/R)^{\frac1s}\}
\end{equation}
so that we have 
\begin{equation}\label{obv}
2^{js} < t/R \quad \mbox{for} \quad j < j_t. 
\end{equation}
Hence
$$
\sum_{0 \leq j <  j_t} 2^{js} \|\D_j \Pi_N u\|_{L^{2}} \lesssim 2^{j_ts}R \lesssim \frac{t}{R} R = t, 
$$
$\tilde \g_{\a,N}$-a.s., therefore
\be\label{eq:21}
\tilde\g_{\a,N}\Big(\sum_{0 \leq j < j_t}  2^{js} \|\D_j P_N u\|_{L^{2}} \geq t\Big)=0\,.
\ee
Thus (\ref{eq:11}) and (\ref{eq:21}) give
\be\label{eq:3}
\tilde\g_{\a,N}(\|\Pi_Nu\|_{H^s}\geq t)\lesssim \tilde\g_{\a,N}(\sum_{j\geq j_t}2^{js}\|\D_j \Pi_Nu\|_{L^{2}}\geq t)\,. 
\ee

Let $c_0>0$ small enough in such a way that
\be\label{eq:sigma} 
\s_j:=c_0(j-j_t+1)^{-2}\,,\quad\sum_{j\in\N}\s_j\leq 1\,.  
\ee
Therefore we can bound
\be\label{eq:41}
\tilde\g_{\a,N}\left(\sum_{j\geq j_t}2^{js}\|\D_j \Pi_Nu\|_{L^{2}}\geq t\right)\leq \sum_{j\geq j_t}\tilde\g_{\a,N}(\|\D_j \Pi_Nu\|_{L^{2}}\geq \s_j2^{-js}t)\,.
\ee
For each term of this sum we have ($\{g_j\}_{j\in\Z}$ indicates a sequence of $\mathcal N(0,1)$ random variables)
\bea
\tilde\g_{\a,N}(\|\D_j \Pi_Nu\|_{L^{2}}\geq \s_j2^{-js}t)&=&\tilde\g_{\a,N}(\|\D_j \Pi_Nu\|^2_{L^2}\geq 2^{-2js}\s^2_jt^{2})\nn\\
&\leq&P\left(\sum_{j\sim 2^j}|g_j|^2\geq 2^{2j(\a-s)}\s^2_jt^{2}\right)\nn\\
&=&P\left(\sum_{j\sim 2^j}(|g_j|^2-1)\geq  2^{2j(\a-s)}\s^2_jt^{2}-2^j\right)\nn\\
&\leq&P\left(\sum_{j\sim 2^j}(|g_j|^2-1)\geq c 2^{2j(\a-s)}\s^2_jt^{2} \right)\nn
\eea
where $c>0$ is a suitable small constant and we used the fact that $2(\a-s) > 1$ in the last inequality 
Then from the Bernstein inequality we get 
$$
P\left(\sum_{j\sim 2^j}(|g_j|^2-1)\geq c 2^{2j(\a-s)}\s^2_jt^{2} \right)\nn\\
\leq e^{-c 2^{2j(\a-s)}\s^2_jt^{2}}
$$
Thus recalling \eqref{Rec1pp} we have arrive to the desired estimate
\bea
\mbox{r.h.s. of \eqref{eq:41}}
&\leq& \sum_{j > j_t} e^{-c 2^{2j(\a-s)}\s^2_jt^{2}}\nn\\
&\leq&e^{-c2^{j_t(\a-s)}\s^2_{j_t} t^{2}}\nn\\
&\leq& \exp\left(-c(R)t^{\frac{2\a}{s}}\right)\,.
\eea
\end{proof}

\begin{lemma}\label{lemma:subexpFL}
There is $c>0$ such that
\be\label{eq:subexpFL}
\tilde\g_{\alpha,N}\left(\|\Pi_Nu\|_{FL^{0,1}}\geq t\right)\leq 2\exp\left(- c \frac{t^{2+2\a}}{R^{2\a}}\right) \,,
\ee
for all $t\gtrsim (\log N)^{2}$ if $\a=1$ or $t\gtrsim N^{1-\a}$ for $\a < 1$. 
\end{lemma}
\begin{proof}
We have
\be
\|u\|_{FL^{0,1}}=\sum_{j\in\N}\sum_{n\sim 2^j}|u(n)|\leq\sum_{j\in\N}\|\D_j u\|_{FL^{0,1}}\,.
\ee

Then 
\be\label{eq:1}
\tilde\g_{\a,N}(\|\Pi_Nu\|_{FL^{0,1}}\geq t)\leq \tilde\g_{\a,N}(\sum_{j\in\N}\|\D_j \Pi_Nu\|_{FL^{0,1}}\geq t)\,. 
\ee

We note that
\be\label{eq:holder}
\|\D_j u\|_{FL^{0,1}}\leq 2^{\frac j2}\|\D_j u\|_{L^{2}}\,.
\ee

Let now 
$$
j_t:=\min\{j\in\N\,:\, R2^{\frac{j}{2}}\geq t\}\,,
$$ 
so that we have 
\begin{equation}\label{obv}
R2^{\frac{j}{2}}< t \quad \mbox{for} \quad j < j_t. 
\end{equation}
Therefore using (\ref{eq:holder}) and (\ref{obv}) we get
$$
\sum_{0 \leq j <  j_t} \|\D_j \Pi_N u\|_{FL^{0,1}} \leq 2^{\frac{j_t}{2}}R <t, 
$$
for any element of $B(R)$, therefore
\be\label{eq:2}
\tilde\g_{\a,N}\Big(\sum_{0 \leq j < j_t}  \|\D_j \Pi_N u\|_{FL^{0,1}} \geq t\Big)=0\,.
\ee
Thus (\ref{eq:1}) and (\ref{eq:2}) give
\be\label{eq:3}
\tilde\g_{\a,N}(\|\Pi_Nu\|_{FL^{0,1}}\geq t)\leq \tilde\g_{\a,N}(\sum_{j\geq j_t}\|\D_j \Pi_Nu\|_{FL^{0,p}}\geq t)\,. 
\ee

Let $c_0>0$ small enough in such a way that 
\be\label{eq:sigma}
\s_j:=c_0(j-j_t+1)^{-2}\,,\quad\sum_{j\in\N}\s_j\leq 1\,.  
\ee
We bound
\be\label{eq:4}
\tilde\g_{\a,N}\left(\sum_{j\geq j_t}\|\D_j \Pi_Nu\|_{FL^{0,1}}\geq t\right)\leq \sum_{j\geq j_t}\tilde\g_{\a,N}(\|\D_j \Pi_Nu\|_{FL^{0,1}}\geq \s_jt)\,.
\ee
We introduce the centred sub-Gaussian random variables
$$
Y_j:=\|\D_j \Pi_Nu\|_{FL^{0,1}}-E_\a[\|\D_j \Pi_Nu\|_{FL^{0,1}}]=\sum_{n\sim 2^j}\frac{|g_n|-E[|g_n|]}{\meanv{n}^\a}\, 
$$
and we have
\be\label{eq:rhsof}
\tilde\g_{\a,N}(\|\D_j \Pi_Nu\|_{FL^{0,1}}\geq \s_jt)=\tilde\g_{\a,N}(Y_j\geq \s_jt-E[\|\D_j \Pi_Nu\|_{FL^{0,1}}])\,. 
\ee
Note that
\be
E_\a[\|\D_j \Pi_Nu\|_{FL^{0,1}}]=\sum_{n\sim2^j}\frac{E[|g_n|]}{\meanv{n}^\a}
\simeq \sum_{n\sim2^j}\frac{1}{\meanv{n}^\a}\,.
\ee
We set
$$
s_j:=\s_jt-E_\a[\|\D_j \Pi_Nu\|_{FL^{0,1}}]\,.
$$
It follows that there is $C>0$ large enough, such that for all $t>C(\log_2N)^{2}$ for $\a=1$ or for $t>CN^{1-\a}$ for $\a\neq1$
\be
\inf_{j\in\{j_t\ldots,\log N\}}s_j>0\,. 
\ee
We have by the Hoeffding inequality
\be
\g_{\a,N}(Y_j\geq s_j)\leq 2\exp(- c 2^{\a j}s_j^2) 
\ee
Thus
\bea
\text{r.h.s. of }\eqref{eq:rhsof}&\leq& \sum_{j\geq j_t} 2\exp(- c 2^{\a j}s_j^2) \leq 2\exp(- c 2^{\a j_t}s_{j_t}^2)\nn\\
&\leq& 2\exp\left(- c \frac{t^{4}}{R^2}\right)\,.
\eea
\end{proof}

In the remaining part of this section it is convenient to shorten
\be\label{eq:def-GN}
G_N:=\frac{d}{dt} H[\Pi_N \Phi^N_t (u)]\big|_{t=0}\,. 
\ee

We have
\begin{lemma}\label{lemma:wick}
It holds for all $M<N\in\N$
\be\label{eq:wick}
\|G_N-G_M\|_{L^p(\tilde\g_\a)}\lesssim\frac{p^3}{M^{2\a-1}}\,. 
\ee
\end{lemma}
\begin{proof}

We will prove
\be\label{eq:wick2}
\|G_N-G_M\|_{L^2(\tilde\g_\a)}\lesssim\frac{1}{M^{2\a-1}}\,. 
\ee
The the assertion will follow by the standard hyper-contractivity estimate \cite[Theorem 5.10, Remark 5.11]{janson}, noting that $G_N$ is a multilinear form of at most $6$ factors.

We use formula \eqref{eq:outdt} in the proof of Lemma \ref{lemma:dtE1}. The first summand is bounded in the support of $\tilde\g_\a$ by $R^4$. For the  second and the third summand we estimate the $L^2(\g_\a)$ norm and use that it controls the $L^2(\tilde\g_\a)$ norm. 

We have
\be
\|G_N-G_M\|_{L^2(\tilde\g_\a)}\leq \|\|\Pi_Nu\|^4_{L^4}-\|\Pi_Mu\|^4_{L^4}\|_{L^2(\g_\a)}+\|\{\|\Pi_Nu\|^4_{L^4},{\mathcal F}_N\}-\{\|\Pi_M u\|^4_{L^4},{\mathcal F}_M \}\|_{L^2(\g_\a)}\,.
\ee

By a direct calculation one has that for any $N \in \N$, 
\be\label{formula fourier L4 poisson FN}
\begin{aligned}
& \{\|\Pi_Nu\|^4_{L^4},{\mathcal F}_N\} \\
& =\sum_{\substack{|n_i|,|m_i|\leq N,\,\\\sum_{i=1}^{3}n_i=\sum_{i=1}^{3}m_i} } {\mathtt c}\Big(n_1, n_2, n_3, m_1,m_2, m_3\Big) u(n_1)u(n_2)u(n_3)\overline u(m_1)\overline u(m_2)\overline u(m_3)\,, 
\end{aligned}
\ee
where the coefficients ${\mathtt c}\Big(n_1, n_2, n_3, m_1,m_2, m_3\Big)$ are such that 
\begin{equation}\label{stima c n1 n6}
| {\mathtt c}\Big(n_1, n_2, n_3, m_1,m_2, m_3\Big)| \lesssim 1, \quad \forall n_1, n_2, n_3, m_1,m_2, m_3 \in \Z\,. 
\end{equation}
In what follows we shall use the Wick formula for expectation values of multilinear forms of Gaussian random variables \cite[Theorem 1.28]{janson} in the following form. 
Let $\ell \in \N$ and $S_{\ell}$ be the 
symmetric group on $\{1,\dots,\ell\}$, whose elements are denoted by $\s$. 
We have
\begin{align}\label{eq:Wick}
E_\a\Big[ \prod_{j=1}^{\ell}  u(n_j)  \bar u(m_j)  \Big] 
& = 
\sum_{\sigma \in S_{\ell}}\prod_{j=1}^{\ell} \frac{\d_{m_j,n_{\sigma(j)}}}{1 + |n_j|^{2\a}} 
\\ \nonumber
& \simeq
\sum_{\sigma \in S_{\ell}}\prod_{j=1}^{\ell} \frac{\d_{m_j,n_{\sigma(j)}}}{\meanv{n_j}^{2\a}} 
 \, ,
\end{align}
where $\meanv{\cdot} = (1 + |\cdot|^2)^{1/2}$.
We convey that the labels $m_i$ (respectively $n_i$) are associated to the Fourier coefficients of $\bar u$ (respectively $u$). We say that $\sigma$ contracts the pairs of indexes $(m_j,n_{\sigma(j)})$ and we shorten for any $\Omega\subset \Z^{\ell}\times\Z^\ell$
$$
\s(\Omega):=\Omega\cap\{m_i=n_{\s(i)}\,,i=1,\ldots,\ell\}\,,\quad \s\in S_\ell\,.
$$
We also define the set $\bar \Omega$ to be obtained by $\Omega$ swapping the role of $n_i$ and $m_i$ $i=1,\ldots \ell$. 

Let $N>M$ and define for $a,b\in\N$
\be\label{eq:A-NM}
A^{a,b}_{N,M}:= 
\{
   |n_{a,b}|,|m_{a,b}|\leq N,\,\,
 n_a+n_b=m_a+m_b\,,\,\max(|m_{a,b}|,|n_{a,b}|)>M\}\,.
\ee

Squaring and using (\ref{eq:Wick}) with $\ell=4$ we have
\bea
&&\|\|\Pi_Nu\|^4_{L^4}-\|\Pi_Mu\|^4_{L^4}\|_{L^2(\g_\a)}=\sum_{A^{1,2}_{N,M}\times A^{3,4}_{N,M}}E_\a\left[\prod_{i=1}^4u(n_i)\bar u(m_i)\right]\nn\\
&=&\!\!\!\!\! \sum_{\sigma \in S_{4}} \sum_{\s(A^{1,2}_{N,M})\times \s(A^{3,4}_{N,M})} \frac{1}{\meanv{n_1}^{2\a}\meanv{n_2}^{2\a}\meanv{n_3}^{2\a}\meanv{n_4}^{2\a}} \lesssim \frac1{M^{4\a-2}}\,. 
\eea

Similarly
\be
B^{a,b,c}_{N,M}:= 
\{
   |n_{a,b,c}|,|m_{a,b,c}|\leq N,\,\,
 n_a+n_b+n_c=m_a+m_b+m_c\,,\,\max(|m_{a,b,c}|,|n_{a,b,c}|)>M\,m_c\neq n_b,n_c\}\,
\ee
and
\bea
&&\|\{\|\Pi_Nu\|^4_{L^4},{\mathcal F}_N\}-\{\|\Pi_Nu\|^4_{L^4},{\mathcal F}_N\}\|_{L^2(\g_1)}\nn\\
&\lesssim &\sum_{B^{1,2,3}_{N,M}\times \bar B^{4,5,6}_{N,M}}E_\a\left[\prod_{i=1}^6u(n_i)\bar u(m_i)\right]\nn\\
&\lesssim &\sum_{\sigma \in S_{6}}\sum_{\s(B^{1,2,3}_{N,M})\times \s(\bar B^{4,5,6}_{N,M})}\prod_{i=1}^6\frac{1}{\meanv{n_i}^{2\a}}\lesssim\frac1{M^{4\a-2}}\,.\nn
\eea
\end{proof}

\ni Then we can immediately establish the following result.

\begin{proposition}\label{prop:exp0}
There are $C,c>0$ such that
\be\label{eq:conc1}
\g_\a\left(|G_N-G|\geq t\right)\leq Ce^{-ct^{\frac13} N^{\frac{2\a-1}{3}}}\,. 
\ee
\end{proposition}
\begin{proof}
Having bounded all the moments as in Lemma \ref{lemma:wick}, we can bound also the fractional exponential moment
\be\label{eq:exp.fract}
E_\a\left[\exp\left(c|G_N-G_M|^{\frac13}N^{\frac{2\a}{3}-\frac{1}{3}}\right)\right]<\infty\,,
\ee
for a suitable constant $c>0$. 
From (\ref{eq:exp.fract}) we obtain (\ref{eq:conc1}) in the standard way using Markov inequality. 
\end{proof}

\begin{proof}[Proof of Proposition \ref{prop:Lpbound}]
We will prove that there is $c(R)>0$, such that for all $N\in\N\cup\{\infty\}$
\be\label{eq:finaltail}
\tilde\g_{\a,N}(G_N\geq t)\lesssim e^{-c(R)t^{\zeta(\a)}}\,.
\ee
Then the fact that \eqref{eq:Lpestimate} follows from \eqref{eq:finaltail} is standard. 

By Lemma \ref{lemma:dtE1} we see that for $u\in B(R)$
\bea
|G_N|&\leq& R^4+R^2\|\Pi_Nu\|^2_{FL^{0,1}}+R\|\Pi_N u\|^2_{H^{2-2\a}}\|\Pi_N u\|^3_{FL^{0,1}}\nn\\
&\lesssim& C(R)\|\Pi_N u\|^2_{H^{2-2\a}}\|\Pi_N u\|^3_{FL^{0,1}}\,.
\eea
Therefore
\be
\tilde\g_{\a,N}(G_N\geq t)\leq \tilde\g_{\a,N} \Big(\|\Pi_Nu\|_{H^{2-2\a}}^2 \|\Pi_Nu\|^3_{FL^{0,1}} \geq C(R) t \Big)\,.
\ee
Lemma \ref{lemma:subexpFL} yields for $\a=1$
\be\label{eq:pettine}
\tilde\g_{\a,N}(G_N\geq t)\lesssim e^{-c(R)t^{\frac43}}\,,\qquad\text{for}\quad t\gtrsim (\log N)^2\,.
\ee
For $\a<1$ we estimate  
\begin{align}
&\tilde\g_{\a,N}(\|\Pi_Nu\|_{H^{2-2\a}}^2 \|\Pi_Nu\|^3_{FL^{0,1}} \geq C(R) t )
\\ \nonumber
& \leq 
\tilde\g_{\a,N}(\|\Pi_Nu\|^3_{FL^{0,1}}\geq c(R) t^{k_1}) + 
\tilde\g_{\a,N}(\|\Pi_Nu\|^2_{H^{2-2\alpha}}\geq c(R) t^{k_2})
\end{align}
with 
\begin{equation}\label{AssOverK}
k_1 + k_2 =1.
\end{equation} Then using Lemma  \ref{lemma:subexpFL} we bound
$$
\tilde\g_{\a,N}(\|\Pi_Nu\|^3_{FL^{0,1}}\geq t^{k_1}) \lesssim 
\exp\left(- c(R) t^{\frac{(2+2\a)k_1}{3}}\right)\,,\qquad\text{for}\quad t\gtrsim N^{1-\a}\,
$$
and using Lemma \ref{lemma:subexpHs} we have
$$
\tilde\g_{\a,N}(\|\Pi_Nu\|^2_{H^{2-2\alpha}}\geq t^{k_2}) \lesssim 
\exp\left(- c(R) t^{\frac{\a k_2}{2 - 2 \a}}\right)\,.\,
$$
We optimise choosing $k_1, k_2$ such that
$$
\frac{\a k_2}{2 - 2 \a} = \frac{(2+2\a)k_1}{3}.
$$
Together with \eqref{AssOverK}, this leads to the choice 
$$k_2 =  \frac{4 - 4\a^2}{4 - 4\a^2 + 3\a} .$$
It is clear that $k_2 \in (0,1)$ when $\a \in (0,1)$,  thus this choice is admissible.  We finally arrive to
\be\label{eq:pettine2}
\tilde\g_{\a,N}(G_N\geq t)\leq 
\exp\left(- c(R) t^{\frac{2 \a (\a + 1)}{4 - 4\a^2 + 3\a}} \right) \,.
\ee

Note that 
\begin{equation}\label{AlphaCondition0.9}
\frac{2 \a (\a + 1)}{4 - 4\a^2 + 3\a} > 1 \quad \mbox{for} \quad \alpha >   \frac{1+ \sqrt{97}}{12}\,.
\end{equation}
Note also that in order to use Lemma \ref{lemma:subexpHs} we must have $2-2\a < \a - 1/2$, namely $\alpha >5/6$.
This is compatible with \eqref{AlphaCondition0.9}.

When  $\a\neq1$, $t\leq N^{1-\a}$. We set $T:=\lfloor t^\frac{1}{1-\a}\rfloor$. We bound
\be\label{eq>sommandi}
\tilde\g_{\a,N}(|G_N|\geq t)\leq \tilde\g_{\a,N}(|G_T|\geq t/2)+\tilde\g_{\a,N}(|G_N-G_T|\geq t/2)\,. 
\ee
Since $t\geq T^{1-\a}$ the first term can be estimated again by Lemma \ref{lemma:subexpFL}:
\be\label{eq:pettine1}
\tilde\g_{\a,N}(|G_T|\geq t/2)\lesssim \exp\left(- c(R) t^{\frac{2 \a (\a + 1)}{4 - 4\a^2 + 3\a}} \right)\,. 
\ee
For the second summand of (\ref{eq>sommandi}) we observe that $N>T$, so Proposition \ref{prop:exp0} gives
\be\label{eq:pettine3}
\tilde\g_{\a,N}(|G_N-G_T|\geq t/2)\leq \g_{\a,N}(|G_N-G_T|\geq t/2)\leq Ce^{-ct^{\frac13}T^{\frac{2\a-1}3}}
\leq Ce^{-c t^{\frac13+\frac{2\a-1}{3(1-\a)}}}\,. 
\ee

Combining that with \eqref{eq:pettine2}, \eqref{eq>sommandi}, \eqref{eq:pettine1} gives \eqref{eq:finaltail} for $\a\neq1$.

Finally we take $\a=1$. Consider $t\leq N^\e$ for some $\e>0$, set $T:=\lfloor t^\frac{1}{\e}\rfloor$. We bound again as in \eqref{eq>sommandi}. 
Since $t\geq T^{\e}$
\be\label{eq:pettine11}
\tilde\g_{\a,N}(|G_T|\geq t/2)\lesssim 2\exp\left(- c(R) t^{\frac{4}{3}}\right)\,. 
\ee
Since $N>T$, by Proposition \ref{prop:exp0} we get
\be\label{eq:pettine3}
\tilde\g_{\a,N}(|G_N-G_T|\geq t/2)\leq \g_{\a,N}(|G_N-G_T|\geq t/2)\leq Ce^{-ct^{\frac13}T^{\frac{2\a-1}3}}\leq Ce^{-c(R)t^{\frac13+\frac{1}{3\e}}}\,. 
\ee

Combining that with \eqref{eq:pettine}, \eqref{eq:pettine11} and taking $\e$ small enough gives \eqref{eq:finaltail} for $\a=1$.
\end{proof}

The next result also follows by the considerations in this section.

\begin{proposition}\label{prop:exp-mom-L4}
Let $\a>3/4$ and $\l\in\R$. 
The quantity $E_\a[1_{\{\|\Pi_N u\|_{L^2}\leq R\}}e^{\l\|\Pi_Nu\|^4_{L^4}}]$ is bounded uniformly in $N$. 
\end{proposition}
\begin{proof}
Set
$$
L_N:=\begin{cases}
N^{1-\a}&\a\neq1\\
(\log N)^2&\a=1\,.
\end{cases}
$$
We write
\bea
E_\a[1_{\{\|\Pi_N u\|_{L^2}\leq R\}}e^{\l\|\Pi_Nu\|^4_{L^4}}]
&=&\int_0^{\infty}dte^{t\l}\tilde\g_{\a,N}\left(\|\Pi_Nu\|^4_{L^4}\geq t\right)\nn\\
&=&1+\int_1^{L_{\a,N}}dt e^{t\l}\tilde\g_{\a,N}\left(\|\Pi_Nu\|^4_{L^4}\geq t\right)\label{eq:eis}\\
&+&\int_{L_{\a,N}}^{\infty}dt e^{t\l}\tilde\g_{\a,N}\left(\|\Pi_Nu\|^4_{L^4}\geq t\right)\,.\label{eq:zwoi}
\eea
By Lemma \ref{lemma:subexpFL}
\bea
\eqref{eq:zwoi}&\leq& \int_{L_{\a,N}}^{\infty}dte^{t\l}\tilde\g_{\a,N}\left(R^2\|\Pi_Nu\|^2_{FL^{0,1}}\geq t\right)\nn\\
&\leq& 2\int_{L_{\a,N}}^{\infty}dt\exp\left(\l t- c(R) t^{1+\a}\right)<\infty. 
\eea
Set $T:=\lfloor t^\frac{1}{1-\a}\rfloor$ for $\a\neq1$ or $T:=\lfloor t^\frac{1}{\e}\rfloor$ for some $\e>0$ if $\a=1$. We bound
\bea
\eqref{eq:eis}&\leq& 1+\int_1^{L_{\a,N}}dte^{\l t}\tilde\g_{\a,N}\left(\|\Pi_Tu\|^4_{L^4}\geq t/2\right)\label{eq:dri}\\
&+&\int_1^{L_{\a,N}}dte^{\l t}\tilde\g_{\a,N}\left(\left|\|\Pi_Nu\|^4_{L^4}-\|\Pi_Tu\|^4_{L^4}\right|\geq t/2\right)\label{eq:fier}\,. 
\eea

Since $t\geq L_{\a,T}$ again by Lemma \ref{lemma:subexpFL} we have
\be
\eqref{eq:dri}\leq 1+\int_1^{L_{\a,N}}dt\tilde\g_{\a,N}\left(R^2\|\Pi_Nu\|^2_{FL^{0,1}}\geq t\right)\leq 1+2\int_{1}^{\infty}dt\exp\left( \l t-c(R) t^{1+\a}\right)<\infty\,. 
\ee
It remains to bound \eqref{eq:fier}. The same proof of Proposition \ref{prop:exp0} gives for $\a\neq1$
\be
\eqref{eq:fier}\leq C\int_1^{\infty}dte^{\l t-t^{\frac13} T^{\frac{2\a-1}{3}}}\leq C\int_1^{\infty}dte^{\l t-t^{\frac13+\frac{2\a-1}{3(1-\a)}}} <\infty\,
\ee
(since $\a>3/4$) and for $\a=1$
\be
\eqref{eq:fier}\leq C\int_1^{\infty}dte^{\l t-t^{\frac13} T^{\frac{2\a-1}{3\e}}}\leq C\int_1^{\infty}dte^{\l t-t^{\frac13+\frac{2\a-1}{3\e}}}<\infty\,
\ee
for $\e$ sufficiently small (these estimates are loose but sufficient to our end). This ends the proof.
\end{proof}

\section{Quasi-invariance}\label{sect:quasi}

The main result of this section is the following.

\begin{proposition}\label{prop:quasi-inv}
Let $\alpha \in( \bar{\alpha} ,1]$, $\bar{\alpha} := \frac{1+ \sqrt{97}}{12} \sim 0.9$ and $t \in [-1,1]$. There exists $\bar f(t,\cdot)\in L^p(\r_\a)$ for all $p \geq 1$,
such that for any measurable set $A$
\be\label{eq:quasi-inv}
\r_\a(\Phi_t (A))=\int_A\bar f(t,u) \r_\a(du)\,.
\ee
In other words, we have
\be\label{eq:prop_res2}
\bar f(t, \cdot) := \frac{d(\r_\a\circ\Phi_t)}{d\r_\a} \in L^p(\r_\a)\,. 
\ee
for the Radon-Nikodim derivative $\frac{d(\r_\a \circ\Phi_t)}{d\r_\a}$ of $\r_\a\circ\Phi_t$ w.r.t. $\r_\a$.  

\end{proposition}


Let 
$$
\r_{\a, N}(du):=e^{-\frac{\sigma}2\| \Pi_N u\|_{L^4}^4}\tilde\g_{\a,N}(du)\,. 
$$
The convergence $\r_{\a, N}\to\r_\a$ was proven in \cite{LRS, B94}. In particular we have that for any measurable set $A$ for every $\e>0$ there is $\bar N\in\N$ such that for all $N>\bar N$
\be\label{eq"usami}
|\r_\a(A)-\r_{\a, N}(A)|<\e\,. 
\ee


We also set
$\g_{N}^{\perp}$ the Gaussian measure induced by the random Fourier series
$$
\sum_{|n| > N} \frac{g_n}{(1+|n|^{2})^{\frac{1}{2}}}\, e^{inx} \, .
$$
We define the Lebesgue measure on $\C^N \simeq \R^{2N} $ as 
$$
L_N(d \Pi_N u) = \prod_{|n| \leq N} d (\Re u(n))  d (\Im u(n))
$$
using the standard isomorphism between $u$ and its Fourier coefficients.  Since the flow is Hamiltonian we have $L_N(\Pi_N \Phi_1^N (E)) = L_N(E)$.

We start by proving the quasi-invariance of $\r_\a$ w.r.t. the truncated flow, which is defined for all $t$. 
\begin{proposition}\label{prop_res1}
Let $\alpha \in( \bar{\alpha} ,1]$, $\bar{\alpha} := \frac{1+ \sqrt{97}}{12} \sim 0.9$
and recall $\zeta(\a)>1$ defined in (\ref{eq:zeta}). We have  
\begin{equation}\label{QuantQiPreq}
\rho_\a (\Phi^N_t (A)) \lesssim \left( \rho_\a(A) \right)^{1-\varepsilon}
\exp\left( C(R, \varepsilon) \left( 1 + |t|^{\frac{\zeta(\a)}{\zeta(\a) -1}} \right) \right)\,
\end{equation}
for all $N \in \N$ and $\varepsilon >0$.
\end{proposition}

\begin{proof}[Proof of Proposition \ref{prop_res1}]
We compute
\begin{align}
& \r_{\alpha, N}( \Phi_t^N (A))  = 
\int_{\Phi_t^N (A)} e^{-\frac{\sigma}2\| \Pi_N u\|_{L^4}^4} \tilde\g_{\a}(du)  \\ \nonumber
& = \int_{\Phi_t^N (A) \cap B(R)}  L_N (d \Pi_N  u)   \g_{\a, N}^{\perp} (d P_{>N} u)
  \exp \left( - H^{(\a)} [\Pi_N  u]   \right)  
\\ \nonumber
& 
= \int_{A \cap B(R)}  L_N (d \Pi_N  u)   \g_{\a, N}^{\perp} (d P_{>N} u)
  \exp \left( - H^{(\a)} [\Pi_N \Phi_t^N(u)  ]   \right)  \\ \nonumber
\end{align}
Taking the derivative in time and evaluating it at $t =0$ we get
\begin{align}\label{STTG}
\frac{d}{dt} \r_{\alpha, N}( \Phi_t^N (A)) \Big|_{t=0} 
& = -
\int_{A \cap B(R)}  L_N (d \Pi_N  u)   \g_{\a, N}^{\perp} (d P_{>N} u)
  \exp \left( - H^{(\a)} [\Pi_N  u  ]   \right) \left(\frac{d}{dt} H^{(\a)} [\Pi_N \Phi_t^N(u)  ]_{t=0} \right)
\\ \nonumber
& =
- \int_{A }  \rho_{\a,N}(d u)   
G_N (u)
\end{align}
where (recall \eqref{MainIdentity} and \eqref{eq:def-GN})
$$
G_N(u) = - \frac{\sigma}2\|\Pi_Nu\|^4_{L^4}+ \sigma \|\Pi_Nu\|^4_{L^2}+ \frac{\sigma}2 \{\|\Pi_Nu\|^4_{L^4},{\mathcal F}_N\}\,. 
$$
Since $t \in (\mathbb{R}, +) \to \Phi_t^N$
is a one parameter group of transformations, we can easily check that
\begin{equation}\label{EqualityAtT=0}
\frac{d}{d t} \left( \r_{\alpha, N}  \circ \Phi_t^N (A) \right)\Big|_{t=t^*} = 
\frac{d}{d t}\left( \r_{\alpha, N} \circ \Phi_t^N (\Phi_{t^*}^N (A) ) \right) \Big|_{t=0} \, .
\end{equation}
Using \eqref{STTG} and \eqref{EqualityAtT=0} 
under the choice $E= \Phi_{t}^N A$, we arrive to
$$
\frac{d}{d t}  \left( \r_{\a, N} \circ \Phi_t^N (A) \right)\Big|_{t=t^*}
= - \int_{\Phi_{t^*}^N A }  \rho_{\a,N}(d u)   
G_N (u).
$$
Thus combining the H\"older inequality and \eqref{eq:Lpestimate} we get
\begin{equation}\label{QuantQiPreq2}
\left| \frac{d}{d t}  \left( \r_{\a, N} \circ \Phi_t^N (A) \right)\Big|_{t= t^*} \right| 
\leq  \| G_N  \|_{L^{p}} \r_{\a, N}(\Phi_{t^*}^N (A))^{1-\frac{1}{p}} \lesssim_R  
\r_{\a, N}(\Phi_{t^*}^N (A))^{1-\frac{1}{p}} p^{\frac{1}{\zeta(\a)}}\,,
\end{equation}
whence
$$
 \frac{d}{d t}  \left( \left( \r_{\a, N} \circ \Phi_t^N (A) \right) \right)^{1/p}
\lesssim_R p^{\frac{1}{\zeta(\a)}-1}\,,   
$$
whence
\be\label{eq:quasi-invN}
\r_{\a, N}(\Phi_t^N (A))\leq \r_{\a, N}(A)\exp\left(p\log(1+ C(R)|t| p^{\frac{1}{\zeta(\a)}-1}(\r_{\a, N}(A))^{-\frac1p})\right).
\ee

We can now show (\ref{QuantQiPreq}). We may and will assume $\r(A)>0$.  
Consider
\be\label{eq:p}
p=\log \left( \frac{1}{2\r(A)} \right)\,,
\ee
and note that
\be\label{eq:p-consequence}
(2\r(A))^{-\frac1p} = e\,. 
\ee
By \eqref{eq"usami} we have that there is $\bar N = \bar N(A)$ such that $\r_{\a, N}(A)\leq 2\r(A)$ for all $N>\bar N$. 
Thus, for sufficiently large $N$ (\ref{eq:quasi-invN}) reads
\bea\label{STG}
\r_{\a, N}(\Phi_t^N (A))&\leq&2 \rho_\a(A)\exp\left(p\log(1+C(R)|t| p^{\frac{1}{\zeta(\a)}-1}(2\rho_\a(A))^{-\frac1p})\right)\nn\\
&=&2\rho_\a(A)\exp\left(p\log(1+C(R)|t| p^{\frac{1}{\zeta(\a)}-1})\right)\nn\\
&\leq&2\rho_\a(A)\exp\left(C(R) |t| p^{\frac{1}{\zeta(\a)}}\right)\,\nn\\
&\leq&2\rho_\a(A)\exp\left(C(R) |t| \left( \ln \left( \frac{1}{2\r(A)} \right) \right)^{\frac{1}{\zeta(\a)}}\right)\nn\\
&\leq& 
2\rho_\a(A)\exp\left(\varepsilon \left( \ln \left( \frac{1}{2\r(A)}  \right) \right)   
\right) \exp \left(  C(R, \varepsilon) |t|^{\frac{\zeta(\a)}{\zeta(\a) -1}}  \right)\nn\\
&\lesssim &
\rho_\a(A)^{1 - \varepsilon}   \exp \left( C(R,\varepsilon)\left( 1 +|t|^{\frac{\zeta(\a)}{\zeta(\a) -1}} \right) \right)\,,
\eea
for all $\varepsilon >0$, where we used \eqref{eq:p-consequence} in the second line and the Young inequality in the penultimate line. Then (\ref{QuantQiPreq}) follows by (\ref{eq"usami}).
\end{proof}

\

For $\a\neq1$ we extend the flow of the Birkhoff map globally in time using the Bourgain probabilistic argument 
\cite{B94}.
 
\begin{proposition}\label{prop:alabourgain}

Let $\alpha \in( \bar{\alpha} ,1)$, $\bar{\alpha} := \frac{1+ \sqrt{97}}{12} \sim 0.9$, 
$s < \alpha - \frac12$ and $t \in [-1,1]$.
Then the Birkhoff flow map (\ref{eq:flomap}) is globally well defined 
for $\r_\a$-almost all initial data. Moreover
there exists $\gamma, c, C >0$  such that   
\begin{equation}\label{ProbabilisticAPriori}
\sup_{t \in [-1, 1]} \|  \Phi_{t} (u) \|_{H^s} \leq K,
\end{equation}
and 
\begin{equation}\label{ProbabilisticAPrioriN}
 \sup_{t \in [- 1, 1]} \| \Phi_t(u) -  \Phi_t^N(u) \|_{H^{s'}} 
\lesssim K N^{- (s - s')}, \qquad 0 \leq s' < s,
\end{equation}
hold
for all $u$ outside an exceptional set of $\r_\a$-measure $\leq C K^2 e^{-cK^\gamma}$.
\end{proposition}
\begin{proof}
Let $K>0$. We partition $[0,1]$ into~$J$ intervals of size at most 
$$\tau_K :=\frac{\mathtt{c}_{\alpha}}{K^2+1},$$ where $\mathtt{c}_{\alpha}$ is given by 
Lemma \ref{LWPlemma}.
Clearly 
 \begin{equation}\label{J}
 J \leq \mathtt{c}_{\alpha}^{-1}  (K^2 +1) +1 . 
 \end{equation}
We take any $s>\min(2-2\a,5/4-\a)$ (note that $s<\a-1/2$ for oour choice of $\a$) and set 
\begin{align}\nonumber
 E_{K,N}
& :=  \Big\{ u \notin  B_{s}(K/2) \Big\} \cup  
\Big\{ u \notin  \Phi^N_{- \tau_K}(B_{s}(K/2)) \Big\}  \cup 
\Big\{  u \notin  \Phi^N_{- 2 \tau_K}(B_{s}(K/2)) \Big\}
\\ \label{Def:Except} 
&
\quad \quad \quad 
\ldots \cup \Big\{  u \notin  \Phi^N_{- (J-1) \tau_K }(B_{s}(K/2)) \Big\}
\cup 
\Big\{  u \notin  \Phi^N_{-J\tau_K}(B_{s}(K/2)) \Big\} \,.
\end{align}
We will show that the $\rho_\a$ measure of these sets vanishes in the limit $K \to \infty$ (and $\t_K\to0$). 

First we record for later use that since the density of $\r_\a$ is in $L^2(\tilde\g_\a)$ (the proof is done by an elementary adaption of \cite[Lemma 3.1]{B94}) we can use the Cauchy-Schwartz inequality to obtain 
$$
\rho_\a (B_{s}(K/2)^C)\lesssim \sqrt{\tilde\g_\a(\|u\|_{H^s}\geq K/2)}  \lesssim e^{-cK^\gamma}, \qquad \gamma >2\,.
$$
where we used Lemma \ref{lemma:subexpHs} for the last bound. 
 
Using \eqref{QuantQiPreq}, we have for some $\e>0$
\begin{align}\label{TTT}
\rho_\a (E_{K,N}) 
&
\leq \sum_{j=0}^{J} \rho_\a( \Phi^N_{- \tau_K}(B_{s}(K/2))^C ) \leq
\sum_{j=0}^{J} (\rho_\a (B_{s}(K/2))^C ))^{1-\e} \lesssim_{R, \varepsilon} J e^{-cK^\gamma} 
\lesssim_{R, \varepsilon} K^2 e^{-cK^\gamma} \, , 
\end{align}
where we used \eqref{J} in the last inequality.

Let now $\{N_K\}_{K\in\N}$ be a diverging sequence and 
\be\label{eq:inclusione-fondamentale}
E := \bigcap_{K \in \N}  E_{K, N_K} \,.
\ee
Using \eqref{TTT} and Proposition \ref{ApproxThm}, we will first show that the flow $ \Phi_t$ is well defined for all 
$t \in [0,1]$ 
and for all initial data in $E_T^C$. Then we will show that 
$ \rho( E_T) = 0$, concluding the proof of the statement
(negative times are covered just by time reversibility).  

Let us consider
\begin{align}\label{eq:setsA}
 E_{K,N}^C
& :=  \Big\{ u \in  B_{s}(K/2) \Big\} \cap  \Big\{ u \in  \Phi^N_{- \tau_K}(B_{s}(K/2)) \Big\}  \cap \Big\{  u \in  \Phi^N_{- 2 \tau_K}(B_{s}(K/2)) \Big\}
\\ \nonumber 
&
\quad \quad \quad 
\ldots \cap \Big\{  u \in  \Phi^N_{- (J-1) \tau_K }(B_{s}(K/2)) \Big\}
\cap 
\Big\{  u \in  \Phi^N_{-J\tau_K}(B_{s}(K/2)) \Big\} \,.
\end{align}Since 
$$
 \Phi^N_{j \tau_K} (E_{K,N}^C) \subset  B_{s}(K/2),
\qquad j=0, \ldots, J+1,
$$
by the group property of the flow we can apply Lemma \ref{LWPlemma} on each
time interval $[j \tau_K, (j+1) \tau_K]$ so that we have  
\begin{equation}\label{fndlsjdngjdskjng}
\sup_{t \in [0, 1]} \sup_{N \in \mathbb{N}} \sup_{u \in  E_{K,N}^C } \|  \Phi_t^N(u) \|_{H^s} \leq K.
\end{equation}
The bound \eqref{fndlsjdngjdskjng} can be extended to times $t \in [-1,1]$ by the reversibility of the flow.
Thus, for all $K >1$ we can pick $N_K$ sufficiently large 
and invoke Proposition~\ref{ApproxThm},
to show that $\Phi_t$ is well defined for times $t \in [-1,1]$ and data in 
$$ 
E^C = \bigcup_{K \in \N} E_{K, N_K}\,.
$$  
On the other hand,
by \eqref{TTT}
we have
$$\rho_\a(E) = \rho_\a \left( \bigcap_{K \in \N}  E_{K, N_K,T} \right) \leq \lim_{K \to \infty} \rho_\a (  E_{K, N_K,T} )  = 0,$$ 
as claimed.

So the flow $\Phi_t$ is defined for $t \in [-1,1]$ for all data outside an exceptional sets $E_{K,N}$ of measure
smaller than $C K^2 e^{-cK^\gamma}$ (recall \eqref{TTT}). Therefore
 \eqref{fndlsjdngjdskjng} implies \eqref{approxProp} by Proposition \ref{ApproxThm} and then \eqref{ProbabilisticAPriori} easily follows. Moreover \eqref{ProbabilisticAPriori} implies the global approximation bound \eqref{ProbabilisticAPrioriN}. Indeed thanks to \eqref{ProbabilisticAPriori} the 
 local approximation bound \eqref{GlobalApproxN} applies for all data outside the exceptional set
 starting by any initial time $t \in [-1,1]$. This proves the~\eqref{ProbabilisticAPrioriN} and concludes the proof.
 \end{proof}

\

\begin{proof}[Proof of Proposition \ref{prop:quasi-inv}]

Proceeding exactly as in the proof of \cite[Lemma 6.2]{BBM1} we obtain that, given any $\varepsilon > 0 $ we can find $\delta = \delta(\varepsilon, R) >0$ such that for all $A$
\be\label{eq:QIN}
\rho_{\a, N}(A) \leq \delta\quad\Rightarrow\quad \r_{\a, N} (\Phi_t^N (A)) \leq \varepsilon
\ee
Next we pass to the limit $N \to \infty$ in this inequality, showing the quasi invariance of the measure $\r_{\a, N}$.
It suffices to consider only compact sets $A\subset B_s(R)$. The argument then extends to Borel sets 
using the inner regularity of the 
Gaussian measure $\tilde\gamma$ and the continuity of the flow map, that is Lemma \eqref{InverseFlowStab} (see \cite[Lemma 8.1]{sigma}).

%

Assume $\rho_\a(A) \leq \delta/2$ with $A$ compact. We have then for all sufficiently small $\delta'$ (that will depend on $A$ and $\delta$)
\begin{equation}\label{FinQi}
\rho_\a(A + B(\delta')) \leq \delta\,.
\end{equation}
Thus by \eqref{eq:QIN}
\be\label{eq:byQIN}
\rho_{\a,N}( \Phi_{t}^N (A + B(\d')))\leq \e\,.
\ee
For $N$ large enough we have 
$$
\rho_\a( \Phi_{t} (A)) \leq  \rho_\a( \Phi_{t}^N (A + B(\d'))) \leq \rho_{\a,N}( \Phi_{t}^N (A + B(\d')))+\e\leq 2\e\,,
$$ 
where we used Lemma \ref{lemma:recall} in the first inequality, (\ref{eq"usami}) in the second one and (\ref{eq:byQIN}) in the last step. Thus absolutely continuity of $\r\circ\Phi_t$ w.r.t $\r$ is proved.
\end{proof}

\section{Density of the transported measure}\label{sect:density}

In this section we show that the density $\bar f$ of Proposition \ref{prop:quasi-inv} can be obtained as a limit of  finite dimensional approximations. 
We define the finite dimensional approximated densities as
\be\label{eq:fN}
f_{N}(s, u) 
:=\exp \left(- \int_0^s \!\!\!  \frac{d}{d\tau}\left( H^{(\a)}[ \Pi_N \Phi_\tau^N(u)]  \right) d\tau \right)\,.
\ee
Following the notation $ \Phi_\tau(u) := \Pi_\infty \Phi_\tau^\infty(u)$ used in the rest of the paper, 
we can write the limit density as
\be\label{eq:fNinfty}
f_{\infty}(s, u) 
:=\exp \left(- \int_0^s \!\!\!  \frac{d}{d\tau}\left( H^{(\a)}[ \Phi_\tau(u)]  \right) d\tau \right)\,.
\end{equation}
The main result of this section is the following.
\begin{proposition}\label{prop:density}
Let $\alpha \in( \bar{\alpha} ,1]$, $\bar{\alpha} := \frac{1+ \sqrt{97}}{12} \sim 0.9$.
The sequence $\{f_{N}\}_{N\in\N}$ defined by (\ref{eq:fN}) converges 
in $L^p(\r_\a)$ to $f_{\infty}(s, u)$ and it holds $f_{\infty}(s, u) := \bar f(s, u)$ for $\r_\a$-almost all $u$, 
where $\bar f(s, \cdot)$ is the transported density from Proposition \ref{prop:quasi-inv}.
\end{proposition}

Combining Proposition \ref{prop:quasi-inv} and Proposition \ref{prop:density} we complete the proof of 
Theorem \ref{TH:main-alpha}. 

To prove Proposition \ref{prop:density}, first we show that this sequence has a limit in $L^p(\r_\a)$.
This is a consequence of the following lemmas. 
\begin{lemma}\label{prop:UI}
Let $\alpha \in( \bar{\alpha} ,1]$, $\bar{\alpha} := \frac{1+ \sqrt{97}}{12} \sim 0.9$.
We have for all $p \in [1, \infty)$
\be\label{eq:UI-p}
\sup_{N\in\N}\|f_{N}\|_{L^p(\r_\a)}<\infty\,. 
\ee
\end{lemma}
\begin{proof}
Let $p \geq 1$.
We write
$$
\|  f_{N} \|_{L^p(\rho_\a)}^p = \int_{0}^{\infty}d t e^{\eta}\rho_\a( f_{N}(s,u)\geq t^{1/p})
$$
and changing variables $t= e^{\eta}$ we can bound
\bea\label{Plug12}
\|  f_{N} \|_{L^p(\rho_\a)}^p &\leq &e + \int_{1}^{\infty}d \eta e^{\eta}\rho_\a( f_{N}(s,u)\geq e^{\frac \eta p})\nn\\
&=&e + \int_{1}^{\infty}d\eta e^{\eta}\rho_\a \left(\int_0^s \!\!\! \frac{d}{d\tau}\left( H^{(\a)}[ \Pi_N \Phi_\tau^N(u)]  \right) d\tau \geq \frac \eta p\right)\,. \label{eq:bathtub}
\eea
Now we note
$$
\left| \int_0^s \!\!\! \frac{d}{d\tau} H^{(\a)}[ \Pi_N \Phi_\tau^N(u)] d\tau \right|
 \leq s \max_{\tau\in[0,s]} \left| \frac{d}{d\tau} H^{(\a)}[ \Pi_N \Phi_\tau^N(u)] \right| =: \, s 
 \left| \frac{d}{d\tau} H^{(\a)}[ \Pi_N \Phi_\tau^N(u)]  \Big|_{\tau=\tau^*} \right|\,, 
$$ 
for some $\tau^* \in [0, s]$.
Therefore 
\be\label{Plug11}
\rho_\a \left(\int_0^s \!\!\!  \frac{d}{d\tau}\left( H^{(\a)}[ \Pi_N \Phi_\tau^N(u)]  \right) d\tau 
\geq \frac{\eta}{p} \right)
\leq \rho_\a \left( \left| \frac{d}{d\tau} H^{(\a)}[ \Pi_N \Phi_\tau^N(u)]  \Big|_{\tau=\tau^*}  \right| \geq \frac{\eta}{ps} \right)\,.\nn
\ee
Let 
$$
A = \left\{ u : \left| \frac{d}{d\varepsilon} H^{(\a)}[\Pi_N \Phi^N_{\varepsilon} (u) ] \Big|_{\varepsilon = 0} \right| > \frac{\eta}{ps} \right\}\,.
$$
Note that if $u \in A$ then $v = \Phi^N_{-\tau*} (u)$ satisfies
\begin{align*}
\frac{d}{d\tau} 
&
H^{(\a)}[ \Pi_N \Phi_\tau^N (v)]  \Big|_{\tau=\tau^*} 
=
\lim_{\varepsilon \to 0}
\varepsilon^{-1} 
\left( H^{(\a)}[\Pi_N \Phi^N_{\varepsilon} \Phi^N_{\tau^*} (v)] - H^{(\a)}[\Pi_N \Phi^N_{\tau^*} (v)] \right) 
\\
&
=
\lim_{\varepsilon \to 0}
\varepsilon^{-1} 
\left( H^{(\a)}[\Pi_N \Phi^N_{\varepsilon} (u)] - H^{(\a)}[\Pi_N u] \right) 
=
\frac{d}{d\varepsilon} H^{(\a)}[\Pi_N \Phi_{\varepsilon}^N (u) ] \Big|_{\varepsilon = 0} \,,
\end{align*}
hence 
$$\Phi_{-\tau*} (A) = 
\left\{ v : \left| \frac{d}{d\tau} H^{(\a)}[ \Pi_N \Phi_\tau^N (v)]  \Big|_{\tau=\tau^*}   \right| \geq \frac{\eta}{ps} \right\}
$$ 
Thus, using 
Proposition
\ref{prop_res1} we can continue the estimate \eqref{Plug11} as follows
$$
\rho_\a \left( \left| \frac{d}{d\tau} H^{(\a)}[ \Pi_N \Phi_\tau^N(u)]  \Big|_{\tau=\tau^*}  \right| \geq \frac{\eta}{ps} \right) 
\lesssim_{\alpha, \varepsilon, R} 
\left( 
\rho_\a \left( \left| \frac{d}{d\varepsilon} H^{(\a)}[ \Pi_N \Phi_\varepsilon^N(u)]  \Big|_{\varepsilon=0}  \right| \geq \frac{\eta}{ps} 
\right) \right)^{1 - \varepsilon}, 
$$
for all $\varepsilon >0$. Using \eqref{eq:finaltail} we have 
$$
\left( 
\rho_\a \left( \left| \frac{d}{d\varepsilon} H^{(\a)}[ \Pi_N \Phi_\varepsilon^N(u)]  \Big|_{\varepsilon=0}  \right| \geq \frac{\eta}{ps} 
\right) \right)^{1 - \varepsilon} \lesssim_{s, \varepsilon} e^{-C(R) \left( \frac{\eta}{ps} \right)^{\zeta(\a)}}.
$$
%

Plugging this into \eqref{eq:bathtub} we have
$$
\|  f_{N} \|_{L^p(\rho_\a)}^p \leq
  e + \int_{1}^{\infty}d\eta e^{\eta -C(R) \left( \frac{\eta}{ps} \right)^{\zeta(\a)} } \lesssim C(R,p,\a, s),
$$
since we have $\zeta(\a) >1$ (recall (\ref{eq:zeta})).
So (\ref{eq:UI-p}) follows.
\end{proof}

We will also need to show that the sequence $f_N(t, \cdot)$ converges in measure, for all $t \in [-1,1]$ and, 
in particular, the limit is 
$f_\infty(t, \cdot)$.


%
\begin{lemma}\label{lemma:ConvMeas}
Let $\alpha \in( \bar{\alpha} ,1)$, $\bar{\alpha} := \frac{1+ \sqrt{97}}{12} \sim 0.9$. 
For all $t \in [-1,1]$ we have that 
$f_{N}(t, \cdot) \to f_{\infty}(t, \cdot)$ in $\rho_\a$-measure as $N \to \infty$. 
\end{lemma}
\begin{proof}
By the continuity of the exponential function it is sufficient to show 
\begin{equation}\label{MeasConv1}
\int_0^s  \frac{d}{d\tau}\left( H^{(\a)}[ \Pi_N \Phi_\tau^N(u)]  \right) d\tau  \to 
\int_0^s  \frac{d}{d\tau}\left( H^{(\a)}[  \Phi_\tau (u)]  \right) d\tau  
\end{equation}
in $\rho_\a$-measure as $N \to \infty$. 

Since 
$$
\begin{aligned}
& \left|  \int_0^s \frac{d}{d\tau}\left( H^{(\a)}[ \Pi_N \Phi_\tau^N(u)]  \right) d\tau 
- \int_0^s \frac{d}{d\tau}\left( H^{(\a)}[ \Phi_\tau (u)]  \right)  d \tau \right| \\ 
&  \leq |s|  \sup_{\tau \in [0,s]} 
\left|\frac{d}{d\tau}\left( H^{(\a)}[ \Pi_N \Phi_\tau^N(u)]  \right)  - \frac{d}{d\tau}\left( H^{(\a)}[  \Phi_\tau (u)]  \right) \right| 
\end{aligned}
$$
we can deduce \eqref{MeasConv1} from
\begin{equation}\label{MeasConv2}
\sup_{\tau \in [-1,1]} \left|\frac{d}{d\tau}\left( H^{(\a)}[ \Pi_N \Phi_\tau^N(u)]  \right)  - \frac{d}{d\tau}\left( H^{(\a)}[  \Phi_\tau (u)]  \right) \right| 
  \to 0  
\end{equation}
in $\rho_\a$-measure as $N \to \infty$. 

We now compute $\frac{d}{d\tau} H^{(\a)}[ \Pi_N \Phi_\tau^N(u)]$. One has  
\be\label{MainIdentityBis}
\begin{aligned}
\frac{d}{d\tau} H^{(\a)}[\Pi_N \Phi^N_\tau (u)]  & = \{ H^{(\alpha)}\,,\, {\mathcal F}_N \} (\Phi^N_\tau (u)) \\
& 
\stackrel{\eqref{MainIdentity}}{=}- \frac{\sigma}2 \|\Pi_N \Phi_\tau^N (u)\|^4_{L^4}+ \sigma \|\Pi_N \Phi_\tau^N(u)\|^4_{L^2}
+ \frac{\sigma}2\{\|\Pi_N \Phi_\tau^N(u)\|^4_{L^4},{\mathcal F}_N \circ \Phi_\tau^N \}\,.
\end{aligned}
\ee

We will consider the three contributions to \eqref{MainIdentityBis} separately. For the first one we must estimate
\begin{align}\label{eq"aligna}
\left|\|\Pi_N \Phi_\tau^N (u)\|^4_{L^4} -
\|\Pi_N \Phi_\tau (u)\|^4_{L^4}\right| = \left|L ( g_N, \ldots, g_N ) - L ( g, \ldots, g ) \right| ,
\end{align}
where we have defined $$L(h_1, h_2, h_3, h_4, ) = \int h_1 h_2 \overline h_3 \overline h_4$$ and 
$$g_N (\tau, u)= \Pi_N \Phi_\tau^N (u), \qquad g =  \Phi_\tau (u).$$
When it does not create confusion will abbreviate $g_N (\tau, u)$ to
$g_N$ in order to simplify the notations.  
 We decompose telescopically
\begin{align*}
L ( g_N, \ldots, g_N ) - L ( g, \ldots, g ) & = 
L ( g_N, \ldots, g_N ) - L ( g, g_N, \ldots, g_N ) 
\\ & +  L ( g, g_N, \ldots, g_N ) -
L ( g, g, g_N  \ldots, g ) 
\\& + \qquad\ldots 
\\&
+ L ( g,g,g, g_N ) - L ( g, \ldots, g )\,.
\end{align*}
We only show how to handle the first one, as the other ones require a similar procedure.
We have 
\begin{align}\label{fdklskdngfksldkng}
| L ( g_N, \ldots, g_N ) & - L ( g, g_N, \ldots, g_N ) |  =
\left| \int (g_N - g) g_N \overline g_N \overline g_N \right|
\\ \nonumber 
&\leq 
\| g_N - g \|_{L^4}  \| g_N \|^3_{L^4} \lesssim
\| g_N - g \|_{H^{1/4}}  \| g_N \|^3_{H^{1/4}} 
\leq \| g_N - g \|_{H^{\frac14}}  \| g_N \|^3_{H^{s}} ,
\end{align}
where we used the Sobolev embedding. 
Here we restrict to $\frac14 \leq s < \alpha - \frac12$, coherently with Proposition \ref{prop:alabourgain}.
Note that for all $\alpha \in [\bar{\alpha},  1]$ we have $\frac14  < \alpha - \frac12$, so the set of possible $s$ is non empty.
Taking $K = N^{\frac{4s-1}{32}}$ in \eqref{ProbabilisticAPriori} and \eqref{ProbabilisticAPrioriN} gives
\begin{equation}
\sup_{t \in [-1, 1]} ( \| g(t, u) \|_{H^s} + \| g(t, u) \|_{H^s}) \leq 2 N^{\frac{4s-1}{32}},
\end{equation}
and
\begin{equation}
 \sup_{t \in [- 1, 1]} \| g(t,u) - g_N(t,u) \|_{H^{\frac14}} 
\lesssim N^{- \frac{7}{32}(4s -1)}, \quad s> \frac14 \,, 
\end{equation}
for $u$ outside an exceptional set of $\r_\a$-measure smaller than $C N^{\frac{4s-1}{16}} e^{-cN^{\gamma \frac{4s-1}{16}}}$.
The two displays above combined with the \eqref{fdklskdngfksldkng} imply that 
\be
\eqref{eq"aligna}\leq \frac{1}{N^{\frac{2s-1}{8}}}
\ee
with probability at least $1-C N^{\frac{4s-1}{16}} e^{-cN^{ \gamma \frac{4s-1}{16}}}$, that is it converges in measure to zero as $N \to \infty$.

The analysis of the second contribution is similar (actually easier since we simply need to control the $L^2$ norm of the evolution rather than the $L^4$). 

For the last contribution we must control $L ( g_N, \ldots, g_N ) - L ( g, \ldots, g )$ where now we redefine
$$
\begin{aligned}
& L(h_1, h_2, h_3, h_4, h_5, h_6 ) := 
\{\|\Pi_Nu\|^4_{L^4},{\mathcal F}_N\} \\
& \stackrel{\eqref{formula fourier L4 poisson FN}}{=} \sum_{\substack{|n_i|,|m_i|\leq N,\,\\\sum_{i=1}^{3}n_i = \sum_{i=4}^{6}n_i} } \mathtt c(n_1, n_2, n_3, n_4, n_5, n_6) h(n_1)h(n_2)h(n_3)\overline h(n_4)\overline h(n_5)\overline h(n_6)
\end{aligned}
$$
with $| \mathtt c(n_1, n_2, n_3, n_4, n_5, n_6)| \lesssim 1$ (recall \eqref{stima c n1 n6}). We do the same decomposition as before (but of course in this case we have six differences to handle rather than four) 
and we explain how to estimate the first contribution 
\bea
&&\left| L ( g_N, \ldots, g_N ) - L ( g, g_N, \ldots, g_N ) \right| \nn\\
& \lesssim& \sum_{\substack{|n_i|\leq N,\,\\\sum_{i=1}^{3}n_i = \sum_{i=4}^{6}n_i} }
 |g_N(n_1) - g(n_1)| |g_N(n_2)| |g_N(n_3)| | g_N(n_4) | | g_N(n_5)| | g_N(n_6)|  \nn\\ 
&\lesssim& \sum_{\substack{|n_i| \leq N,\,\\\sum_{i=1}^{6}n_i = 0 } }\!\!\!\!\!\! |(g_N(n_1) - g(n_1)| |g_N(n_2)| |g_N(n_3)| 
 | g_N(-n_4) | | g_N(-n_5) | | g_N(-n_6) |  \,.\label{eq:convstru}
\eea
After spotting the convolution structure of \eqref{eq:convstru} and 
recalling the inequality
$$
\| (a_1 * a_2 * a_3 * a_4 * a_5 * a_6)_n \|_{\ell_n^{\infty}} \leq \prod_{j^1}^6 \| (a_j)_n \|_{\ell_n^{\frac65}}
$$
we can further estimate 
$$
\eqref{eq:convstru}\leq \| g_N - g  \|_{FL^{0, \frac65}} \| g \|_{FL^{0, \frac65}}^5 
$$
Then using the inequality (for sequencies) $\| a_n \|_{\ell_n^{\frac65}} \leq \left( \sum_{n \in \Z}\langle n \rangle^{2s} |a_n|^2 \right)^{\frac12} $
valid for $s > 1/3$ we further estimate
\be\label{eq:lastdisplay}
\eqref{eq:convstru} \leq \| g_N(\tau,u) - g(\tau,u)  \|_{H^{1/3 +}} \| g \|_{H^{1/3 +}}^5 
\ee
where here we restrict to $\frac13 < s < \alpha - \frac12$, coherently with Proposition \ref{prop:alabourgain}.
Note that for all $\alpha \in [\bar{\alpha},  1]$ we have $\frac13  < \alpha - \frac12$, so the set of possible $s$ is non empty.
From here, we can proceed exactly as before (from \eqref{fdklskdngfksldkng} onward) to show 
that also this last summand of (\ref{MainIdentityBis}) converges in measure to zero as 
$N \to \infty$.
This implies the convergence in measure of $f_N(t, \cdot)$ to $f_\infty(t, \cdot)$, for all $t \in [-1,1]$
so the proof is concluded. 
\end{proof}

\begin{proof}[Proof of Proposition \ref{prop:density}]
By Lemma \ref{prop:UI} and Lemma \ref{lemma:ConvMeas}  we obtain that, for all $p \geq 1$, the sequence
$f_N$ 
converges in $L^{p}(\rho_\a)$. 
More precisely the uniform $L^p(\rho_\a)$ bounds at a fixed $p$ and the convergence in measure of the sequence 
guarantees the convergence in 
$L^{p'}(\rho_\a)$, for all $p' < p$ (see for instance \cite[Lemma 3.7]{BBM1}) to a certain $L^{p'}(\rho_\a)$ function. Moreover, 
this limit  
must coincided $\rho_\a$-a.s. with $f_{\infty}$ by Lemma by \ref{lemma:ConvMeas}.
\end{proof}

\begin{proof}[Proof of Theorem \ref{TH:main-alpha}]
Once we have identified the $L^{p}(\rho_\a)$ limit $f_{\infty}$, in order to complete the proof of Theorem \ref{TH:main-alpha} we need to show 
that $f_{\infty} = \bar{f}$ $\rho_\a$-a.s., where we recall that $\bar{f}$ is the density of the  
 transport of the measure $\rho_\a$ under the flow. 
The almost sure identity $f=\bar f$ follows by an abstract argument which can be adapted 
line by line from \cite[Proposition 7.2]{BBM1}. 
\end{proof}


\end{document}